\documentclass[onefignum,onetabnum]{siamonline190516}

\usepackage{upgreek}
\usepackage{amsfonts}
\usepackage{graphicx}
\usepackage{overpic}
\usepackage{epstopdf}
\usepackage{algorithmic}
\usepackage{enumerate}
\usepackage{bm}
\ifpdf
  \DeclareGraphicsExtensions{.eps,.pdf,.png,.jpg}
\else
  \DeclareGraphicsExtensions{.eps}
\fi

\graphicspath{{image/}}

\usepackage{verbatim} 

\usepackage{enumitem}
\setlist[enumerate]{leftmargin=.5in}
\setlist[itemize]{leftmargin=.5in}

\newsiamthm{claim}{Claim}
\newsiamremark{conjecture}{Conjecture}
\newsiamremark{remark}{Remark}
\newsiamremark{example}{Example}
\newsiamremark{hypothesis}{Hypothesis}
\newsiamremark{problem}{Problem}
\newsiamremark{assumption}{Assumption}

\usepackage{etoolbox}
\AtEndEnvironment{remark}{\null\hfill$\Diamond$}
\AtEndEnvironment{assumption}{\null\hfill$\Diamond$}

\newcommand{\mcl}{\mathcal}

\newcommand{\mbf}{\mathbf}
\newcommand{\mbb}{\mathbb}

\newcommand{\Div}{\text{div}}

\newcommand{\dd}{{\rm d}}

\newcommand{\bx}{\mbf x}
\newcommand{\bu}{\mbf u}

\newcommand{\bv}{\mbf v}
\newcommand{\bz}{\mbf z}
\newcommand{\by}{\mbf y}

\newcommand{\bw}{\mbf w}

\newcommand{\bo}{\mbf o}

\newcommand{\eps}{\epsilon}

\newcommand{\vphi}{\varphi}
\newcommand{\bphi}{{\bm{\phi}}}
\newcommand{\bpsi}{{\bm{\psi}}}

\newcommand{\bvphi}{{\bm{\vphi}}}

\newcommand{\Md}{M_\Omega}
\newcommand{\cc}{K}
\renewcommand{\S}{\widetilde{\C}}
\newcommand{\tK}{\widetilde{K}}

\newcommand{\mN}{\mcl{N}}

\newcommand{\mH}{\mcl{H}}
\newcommand{\mK}{\mcl{K}}
\newcommand{\mL}{\mcl{L}}

\newcommand{\mB}{\mcl{B}}

\newcommand{\mU}{\mcl{U}}
\newcommand{\mV}{\mcl{V}}
\newcommand{\mP}{\mcl{P}}
\newcommand{\mA}{\mcl{A}}

\newcommand{\tL}{\widetilde{L}}
\newcommand{\tphi}{\widetilde{\phi}}

\newcommand{\tbphi}{\widetilde{\bphi}}
\newcommand{\tN}{\widetilde{N}}
\newcommand{\tz}{\widetilde{z}}
\newcommand{\tbz}{\widetilde{\bz}}
\newcommand{\tchi}{\widetilde{\chi}}
\newcommand{\tTheta}{\widetilde{\Theta}}

\newcommand{\R}{\mbb{R}}
\newcommand{\EE}{\mbb{E}}

\newcommand{\Law}{{\rm Law}}

\newcommand{\st}{{\rm\,s.t.}}

\newcommand{\C}{\mathcal{K}}
\newcommand{\K}{\mathcal{K}}

\newcommand{\hP}{\overline{P}}
\newcommand{\hB}{\overline{B}}
\newcommand{\hF}{\overline{F}}

\newcommand{\cC}{\mathcal{C}}

\newcommand{\tr}{\operatorname{tr}}

\definecolor{darkred}{rgb}{.7,0,0}

\definecolor{darkgreen}{rgb}{.15,.55,0}

\definecolor{darkblue}{rgb}{0,0,0.7}

\usepackage{amsopn}

\DeclareMathOperator*{\minimize}{{\rm minimize}}

\newtheorem{innercustomexample}{Example}
\newenvironment{runningexample}[1]
  {\renewcommand\theinnercustomexample{#1}\innercustomexample}
  {\endinnercustomexample}

  \AtEndEnvironment{example}{\null\hfill$\Diamond$}
  \AtEndEnvironment{runningexample}{\null\hfill$\Diamond$}

\newcommand{\yc}[1]{{#1}}
\newcommand{\bh}[1]{{#1}}
\newcommand{\ho}[1]{{#1}}
\newcommand{\as}[1]{{#1}}

\reversemarginpar
\usepackage[colorinlistoftodos,prependcaption,textsize=tiny]{todonotes}

\headers{Solving and learning nonlinear {PDEs} with {GPs}}{Y. Chen, B. Hosseini, H. Owhadi, AM. Stuart}

\title{Solving and Learning Nonlinear PDEs with Gaussian Processes}

\author{Yifan Chen \and Bamdad Hosseini \and Houman Owhadi\footnote{Corresponding author.} \and  Andrew M Stuart
  \thanks{Computing and Mathematical Sciences, Caltech, Pasadena, CA
  (\email{yifanc@caltech.edu},
\email{bamdadh@caltech.edu}, \email{owhadi@caltech.edu}, \email{astuart@caltech.edu}).}}

\ifpdf
\hypersetup{
  pdftitle={KF-PDE},
  pdfauthor={Y. Chen, B. Hosseini, H. Owhadi and A. Stuart}
}
\fi

\begin{document}

\maketitle

\begin{abstract}
We introduce
a simple, rigorous, and unified framework for solving
nonlinear partial differential equations (PDEs), and for solving inverse problems (IPs) involving the identification of parameters in PDEs, using the framework of Gaussian processes. The proposed \as{approach: (1) provides a natural generalization of collocation kernel methods to nonlinear PDEs and IPs; (2) has guaranteed convergence for a very general class of PDEs, and comes equipped  with a path to compute error bounds for specific PDE approximations; (3)} inherits the state-of-the-art computational complexity of linear solvers for dense kernel matrices.
The main idea of our method is to approximate the solution of a given PDE as the \as{maximum a posteriori (MAP) estimator of a Gaussian process conditioned on solving the PDE at a finite number of collocation points.} Although this optimization problem is infinite-dimensional, it can be reduced to a finite-dimensional one by introducing additional variables corresponding to the values of the derivatives of the solution at collocation points; this generalizes
the representer theorem arising in Gaussian process regression.
\as{The reduced optimization problem has the form of a quadratic objective function subject to nonlinear constraints; it is solved with a variant of the Gauss--Newton method. The resulting
algorithm (a) can be interpreted as solving successive linearizations of the nonlinear PDE,
and (b) in practice is found  to converge  in a small number  of iterations (2 to 10), for a wide  range of PDEs.
Most traditional approaches to IPs interleave parameter updates with numerical solution of the PDE; our algorithm
solves for both parameter and PDE solution simultaneously.} Experiments on nonlinear elliptic PDEs, Burgers' equation, a regularized Eikonal equation, and an IP for permeability identification in Darcy flow
illustrate the efficacy and scope of our framework.
\end{abstract}



\section{Introduction}\label{sec:introduction}

Two hundred years ago, modeling a physical problem and solving the underlying differential equations would have required the expertise of Cauchy or Laplace, and it would have been done by hand through a laborious process of discovery.
Sixty years ago, the resolution of the underlying differential equation would have been \as{addressed} using computers, but  modeling
and design of the solver would have still been done by hand.
Nowadays, there is increasing interest in automating these steps by casting them as machine learning problems.
The resulting methods can be divided into two main categories: (1)
methods based on variants of artificial neural networks (ANNs) \cite{goodfellow2016deep}; and (2) methods based on kernels and Gaussian Processes (GPs) \cite{williams1996gaussian,scholkopf2018learning}. In the context of (1)
there has been recent activity toward solving nonlinear PDEs, whilst the systematic development of  methods of type (2) for nonlinear PDEs has remained largely open. However, methods of type
(2) hold potential for considerable advantages over methods of type (1), both in terms of theoretical
analysis and numerical implementation. In this paper, our goal is to develop
 \as{a simple kernel/GP framework for solving  nonlinear PDEs and related inverse problems (IPs); in particular the methodology
 we introduce has  the following properties:}
 \begin{itemize}
     \item \as{the proposed collocation setting for solving nonlinear PDEs and IPs is a direct generalization of optimal recovery kernel methods for linear PDEs} \cite{Owhadi:2014, owhadi2017multigrid,  owhadi2019operator}, and a natural generalization of radial basis function collocation methods \cite{zhang2000meshless, wendland2004scattered}, and of meshless kernel methods \cite{schaback2006kernel};
     \item theoretically, the proposed method is provably convergent and amenable to rigorous numerical analysis,
     \as{suggesting new research directions to
     generalize the analysis of linear regression methods \cite{wendland2004scattered} to the proposed collocation setting for solving nonlinear PDEs};
     \item computationally, it inherits the complexity of state-of-the-art solvers for dense kernel matrices,
     \as{suggesting new research to
     generalize the work of \cite{schafer2020sparse}, which
     developed optimal approximate methods for linear regression,
     to the proposed collocation setting for solving nonlinear PDEs and IPs};
     \item \as{for IPs the methodology is closely aligned
     with methodologies prevalent in the PDE-constrained
     optimization literature \cite{hinze2008optimization}
     and suggests the need for new computational and
     theoretical analyses generalizing the standard
     optimization and Bayesian approaches found in
     \cite{stuart-bayesian-lecture-notes, engl1996regularization,kaipio2006statistical}.}
 \end{itemize}
Since ANN methods can also be interpreted as kernel methods with kernels learned from data \cite{jacot2018neural, owhadi2020ideas, wilson2016deep}, our framework could also be used for theoretical analysis of such methods.

In Subsection~\ref{sec:problem-setup} we summarize the theoretical foundations and numerical implementation of our method in the context of a simple nonlinear elliptic PDE. In Subsection~\ref{sec:relevant-lit} we give
a literature review, placing the proposed methodology in the context of other research at the
intersection of machine learning and PDEs. The outline of the article is given in Subsection~\ref{sec:outline}.

\subsection{Summary of the Proposed Method}\label{sec:problem-setup}
For demonstration purposes, we summarize the key ideas of our method for solving a nonlinear second-order elliptic PDE. This PDE will also serve as a running example in Section~\ref{sec:PDEs} where we present an abstract framework for general nonlinear PDEs.

Let $d \geq 1$ and $\Omega$ be a bounded open domain in $\R^d$ with a \bh{Lipschitz boundary $\partial \Omega$}.
Consider the following nonlinear elliptic equation for $u^\star$:
\begin{equation}
  \label{elliptic-proto-PDE}
  \left\{
    \begin{aligned}
      -\Delta u^\star(\bx) + \tau\big(u^\star(\bx)\big) &= f(\bx), &&\forall
      \bx \in \Omega\,, \\
      u^\star(\bx) & = g(\bx), &&\forall
      \bx \in \partial\Omega\, ,
    \end{aligned}
    \right.
\end{equation}
where $f:\Omega \to \R, \bh{g: \partial \Omega \to \R}$ and $\tau: \R \to \R$ are given continuous functions.
We assume that $f,g, \tau$ are chosen appropriately so that the PDE has a unique classical solution
(for abstract theory of nonlinear elliptic PDEs see for example \cite{ruadulescu2008qualitative, smoller2012shock}).
In Subsection~\ref{sec:intro:implementation} we will present a concrete numerical experiment
where $\tau(u) = u^3$ and $g(\bx) =0$.

\subsubsection{Optimal Recovery}\label{sec:intro:mapest}
\as{Our proposed method starts with an optimal recovery problem that can also be interpreted as \as{maximum a posterior} (MAP) estimation for  a GP constrained by a PDE}. More precisely, consider a nondegenerate, symmetric, and positive definite  kernel function
$\cc: \overline{\Omega} \times \overline{\Omega} \rightarrow \R$ \bh{where $\overline{\Omega} := \Omega \cup \partial \Omega$}.
Let $\mU$
be the RKHS associated with $\cc$ and denote the RKHS norm by $\| \cdot \|$. Let $1\leq \Md < M  <\infty$ and  \bh{fix} $M$ points in $\overline{\Omega}$ such that $\bx_1,\ldots, \bx_{M_\Omega} \in \Omega$ and $\bx_{\Md+1},\ldots,\bx_M\in \partial \Omega$.
Then, our method approximates the solution $u^\star$ of
\eqref{elliptic-proto-PDE} with a minimizer $u^\dagger$ of the following optimal recovery problem:
  \begin{equation}\label{running-example-optimization-problem}
    \left\{
      \begin{aligned}
  &\minimize_{u \in \mU}~\| u\| && \\
  &\st \quad -\Delta u(\bx_m)+ \tau\big(u(\bx_m)\big)=f(\bx_m), \quad &&\text{for } m=1,\ldots,\Md\,,\\
 & \hspace{6ex} u(\bx_m)=g(\bx_m),     \quad &&\text{for }m=\Md+1,\ldots,M\,.
\end{aligned}
\right.
\end{equation}
 Here, we assume $\cc$ is chosen appropriately so that $\mU \subset C^2(\Omega)\cap C(\overline{\Omega})$, which ensures the pointwise PDE constraints in \eqref{running-example-optimization-problem}
  are well-defined.

  A minimizer $u^\dagger$ can be interpreted as a MAP estimator of a GP $\xi \sim \mN(0,\mathcal{K})$\footnotemark
  (where $\mathcal{K}$ is the integral operator with kernel $K$)
  conditioned on
  satisfying the PDE at the collocation points $\bx_m, 1\leq m \leq M$.
  \as{Such a view has been introduced for solving linear PDEs in  \cite{Owhadi:2014, owhadi2017multigrid} and a closely related approach is studied in \cite[Sec.~5.2]{cockayne2016probabilistic}; the methodology
  introduced via \eqref{running-example-optimization-problem} serves as a prototype for generalization to nonlinear PDEs. Here it is important to note that in the nonlinear case the conditioned GP is no longer Gaussian in general; thus the solution of
  \eqref{running-example-optimization-problem} is a MAP estimator only and is not the conditional expectation, except in the
  case where $\tau(\cdot)$ is a linear function.}
  \footnotetext{{This Gaussian prior notation is equivalent to the GP notation $\mathcal{GP}(0,K)$, where $K$ is the covariance function. See further discussions in Subsection \ref{sec:constructing-theta}.}}

In the next subsections, we show equivalence of \eqref{running-example-optimization-problem} and a finite
dimensional constrained optimization problem \eqref{eqn: finite dim optimization for elliptic eqn}.
This provides existence\footnote{Uniqueness of a minimizer is not necessarily a consequence of the assumed
uniqueness of the solution to the PDE \eqref{elliptic-proto-PDE}. We provide additional discussions of uniqueness results in Subsection \ref{sec-uniqueness-optimality}.} of a minimizer to \eqref{running-example-optimization-problem}, as
well as the basis for a numerical method to approximate the minimizer, based on
solution of an unconstrained finite-dimensional optimization problem
\eqref{running-example-unconstrained-optimization-problem}.

\subsubsection{Finite-Dimensional Representation}\label{sec:intro:representertheorem}
  The key to our numerical algorithm for solving \eqref{running-example-optimization-problem}
  is a representer theorem that characterizes $u^\dagger$ via a finite-dimensional representation formula.
  To achieve this we first reformulate  \eqref{running-example-optimization-problem}
 as a two level optimization problem:
   \begin{equation}\label{running-example-two-level-optimization-problem}
     \left\{
     \begin{aligned}
       & \minimize_{\bz^{(1)} \in \R^M, \bz^{(2)} \in \R^{M_\Omega}} \: \left\{
       \begin{aligned}
          & \minimize_{u \in \mU} ~\|u\|\\
          &\st \quad  u(\bx_m)=z^{(1)}_m \text{ and } -\Delta u(\bx_m)=z^{(2)}_m,\text{ for } m=1,\ldots,M,
       \end{aligned} \right.\\
  &\st \quad z^{(2)}_m +\tau(z^{(1)}_m)=f(\bx_m), \hspace{28ex} \text{for } m=1,\ldots,\Md\,, \\
  &     \hspace{6ex}   z^{(1)}_m=g(\bx_m),  \hspace{37.5ex} \text{for } m=\Md+1,\ldots, M\,.
\end{aligned}
\right.
\end{equation}
 Denote $\phi^{(1)}_m=\updelta_{\bx_m}$
  and $\phi^{(2)}_m = \updelta_{\bx_m} \circ (- \Delta )$, where $\updelta_{\bx}$ is the
  Dirac delta function centered at $\bx$. We further use the shorthand notation
  $\bphi^{(1)}$ and $\bphi^{(2)}$ for the $M$ and $M_\Omega$-dimensional vectors with entries $\phi^{(1)}_m$ and $\phi^{(2)}_m$
  respectively,
  and  $\bphi$  for the $(M+M_\Omega)$-dimensional vector  obtained by concatenating $\bphi^{(1)}$ and $\bphi^{(2)}$.
  Similarly, we write $\bz$ for the $(M+M_\Omega)$-dimensional vector concatenating $\bz^{(1)}$ and $\bz^{(2)}$.

  For a fixed $\bz$, we can solve the first level optimization problem explicitly due to the  representer theorem\footnote{This is not the standard RKHS/GP representer theorem \cite[Sec.~2.2]{williams1996gaussian} in the sense that measurements include the pointwise observation of higher order derivatives of the underlying GP. See \cite{kimeldorf1971some} and \cite[p.~xiii]{wahba1990spline} for related spline representation formulas with derivative information.} (see  \cite[Sec.~17.8]{owhadi2019operator}), which leads to the conclusion that
\begin{equation}\label{running-example-minimizer-representer-form}
u(\bx)=\cc(\bx,\bphi) \cc(\bphi,\bphi)^{-1} \bz\, ;
\end{equation}
\as{here $\cc(\cdot,\bphi)$ denotes the $(M+M_\Omega)$-dimensional vector field with entries $\int \cc(\cdot,\bx')\phi_m(\bx')\,\dd \bx'$}
and $\cc(\bphi,\bphi)$ is the $(M+M_\Omega)\times (M+M_\Omega)$-matrix with entries
$\int \cc(\bx,\bx')\phi_m(\bx)\phi_j(\bx')\,\dd \bx\, \dd \bx'$ with the $\phi_m$ denoting the
entries of $\bphi$.
For this solution, $\|u\|^2=  \bz^T  \cc(\bphi,\bphi)^{-1}\bz $, so we can equivalently formulate \eqref{running-example-two-level-optimization-problem} as
a finite-dimensional optimization problem:
 \begin{equation}
 \label{eqn: finite dim optimization for elliptic eqn}
   \left\{
   \begin{aligned}
  & \minimize_{\bz \in \R^{M + M_\Omega}}  \quad   \bz^T  \cc(\bphi,\bphi)^{-1}\bz \\
  &\st  \quad  z^{(2)}_m +\tau(z^{(1)}_m)=f(\bx_m),  &&\text{for } m=1,\ldots,\Md\,, \\
  & \hspace{6ex} z^{(1)}_m=g(\bx_m), &&\text{for } m=\Md+1,\ldots, M\,.
\end{aligned}
\right.
  \end{equation}
  Moreover, using the equation $ z^{(2)}_m =f(\bx_m)-\tau(z^{(1)}_m)$ and the boundary data,
  we can further eliminate $\bz^{(2)}$ and rewrite it as an unconstrained problem:
   \begin{equation}\label{running-example-unconstrained-optimization-problem}
     \minimize_{\bz^{(1)}_{\Omega} \in \R^{\Md}}~
     \big(\bz^{(1)}_{\Omega}, g(\bx_{\partial\Omega}), f(\bx_{\Omega})-\tau(\bz^{(1)}_\Omega)\big)
     \cc(\bphi,\bphi)^{-1}
     \begin{pmatrix}\bz^{(1)}_{\Omega} \\
       g(\bx_{\partial\Omega}) \\
       f(\bx_{\Omega})-\tau(\bz^{(1)}_{\Omega})\end{pmatrix},
   \end{equation}
   where we used $\bx_{\Omega}$ and $\bx_{\partial\Omega}$ to denote the interior and boundary points
   respectively,
   $\bz^{(1)}_{\Omega}$ denotes the $\Md$-dimensional vector of the $z_i$
   for $i=1, \dots, \Md$ associated to the interior points $\bx_{\Omega}$ while $f(\bx_{\Omega}), g(\bx_{\partial\Omega})$
   and $\tau(\bz^{(1)}_\Omega)$ are vectors obtained  by applying the corresponding functions
   to   entries of their input vectors. For brevity we have suppressed the transpose
   \bh{signs in} the row vector multiplying the matrix from the left in  \eqref{running-example-unconstrained-optimization-problem}.

 The foregoing considerations lead to the following existence result which underpins\footnote{Although this existence result is desirable, it is not necessary in the sense that our variational analysis of  minimizers of \eqref{running-example-optimization-problem}, remains true for near-minimizers. Furthermore, the error estimates of Section \ref{sec:error-estimates} can be modified to include the error incurred by approximating $u^\star$ with a near-minimizer rather than a  minimizer $u^\dagger$.} our numerical
 method for \eqref{elliptic-proto-PDE}; furthermore  \eqref{running-example-unconstrained-optimization-problem} provides the basis for our numerical implementations. We summarize these facts:

 \begin{theorem}
   The variational problem \eqref{running-example-optimization-problem} has a minimizer of the form\\ $u^\dagger(\bx)=\cc(\bx,\bphi) \cc(\bphi,\bphi)^{-1} \bz^\dagger$, where $\bz^\dagger$ is a minimizer of \eqref{eqn: finite dim optimization for elliptic eqn}. Furthermore $u^\dagger(\bx)$ may be found by
   solving the unconstrained minimization problem \eqref{running-example-unconstrained-optimization-problem}
   for $\bz^{(1)}_{\Omega}$.
 \end{theorem}
 \begin{proof}
 Problems \eqref{running-example-optimization-problem}, \eqref{running-example-two-level-optimization-problem} and \eqref{eqn: finite dim optimization for elliptic eqn} are equivalent.
 It is therefore sufficient to show that \eqref{eqn: finite dim optimization for elliptic eqn} has a minimizer. Write $\bz^\star$ for the vector with entries
  $z^{\star(1)}_m=u^\star(\bx_m)$ and  $z^{\star(2)}_m=-\Delta u^\star(\bx_m)$. Since $u^\star$ solves the PDE \eqref{elliptic-proto-PDE}, $\bz^\star$ satisfies the constraints on $\bz$ in
  \eqref{eqn: finite dim optimization for elliptic eqn}. It follows that the minimization in
  \eqref{eqn: finite dim optimization for elliptic eqn}
  can be restricted to the set \as{$\cC$ of $\bz$ that satisfies
  $\bz^T  \cc(\bphi,\bphi)^{-1}\bz \leq (\bz^\star)^T  \cc(\bphi,\bphi)^{-1}\bz^\star $ and the nonlinear constraints.
  The set $\cC$ is compact and non-empty:
  compact because $\tau$ is continuous and so the constraint set is closed as it is the pre-image of a closed set, and the  intersection of a closed set with a compact set is compact; and nonempty because it contains $\bz^*$. Thus}
  the objective function
$\bz^T K(\bphi,\bphi)^{-1} \bz$ achieves its minimum in the set $\cC.$ Once $\bz^{(1)}_{\Omega}$ is
found we can extend to the boundary points to obtain $\bz^{(1)}$, and use the
   nonlinear relation between $\bz^{(1)}$ and $\bz^{(2)}$ to reconstruct $\bz^{(2)}$. This gives
   $\bz^\dagger$.
 \end{proof}

 \subsubsection{Convergence Theory}\label{sec:intro:conv-error}

The consistency of our method is guaranteed by the
 convergence of $u^\dagger$ towards $u^\star$
 as $M$,
 the total number of collocation points, goes to infinity. We first present this result  in the case
 of the nonlinear PDE \eqref{elliptic-proto-PDE} and defer a
 more general version  to
 Subsection~\ref{sec:Conv-theory}. We also give a simple proof of convergence here as it is instructive in
 understanding why the method works and how the more general result can be established.
\bh{Henceforth we use $| \cdot |$ to denote the standard Euclidean norm and write $\mU \subseteq \mH$
to denote  $\mU$ being continuously embedded in Banach space $\mH$.}

\begin{theorem}\label{thm:intro:convergence} Suppose  that
  $\cc$ is chosen so that $\mU \subseteq H^s(\Omega)$   for some $s >2+d/2$ and
 that \eqref{elliptic-proto-PDE} has a unique classical solution $u^\star \in \mU$.
  Write $u^\dagger_M$ for a minimizer of \eqref{running-example-optimization-problem}
  with $M$ distinct collocation points $\bx_1,\ldots,\bx_M$.  Assume that, as $M\rightarrow \infty$,
  \begin{equation*}
      \sup_{\bx\in \Omega} \: \min_{1\leq m \leq M_\Omega}
      |\bx-\bx_m|\rightarrow 0
      \quad \text{and} \quad \sup_{\bx\in \partial \Omega} \: \min_{M_{\Omega+1}\leq m \leq M}
      |\bx-\bx_m|\rightarrow 0\, .
  \end{equation*}
  Then, as $M\rightarrow \infty$, $u^\dagger_M$ converges towards $u^\star$
  pointwise in $\Omega$ and in $H^t(\Omega)$ for any \bh{$t<s$}.
\end{theorem}
\begin{proof}
  The proof relies on a simple compactness argument together with the Sobolev embedding theorem \cite{adams2003sobolev, brezis2010functional}. \yc{First, as $u^\star$ satisfies the constraint in
  \eqref{running-example-optimization-problem} and $u^\dagger_M$ is the minimizer,
  it follows that
  $\|u^\dagger_M\| \leq \|u^\star\|$ for all $M\geq 1$.
  This implies $\|u^\dagger_M\|_{H^s(\Omega)}\leq C \|u^\star\|$ because $\mU$ is continuously embedded into $H^s(\Omega)$. Let $t\in (2+d/2,s)$ so that $H^t(\Omega)$ embeds continuously into $C^2(\Omega) \cap C(\overline{\Omega})$ \bh{\cite[Thm.~4.12]{adams2003sobolev}}.
  Since $H^s(\Omega)$ is compactly embedded into $H^t(\Omega)$, we deduce that there exists a subsequence $\{M_p, p\geq 1\} \subset \mathbb{N}$ and a limit $u^\dagger_\infty \in H^t(\Omega)$ such that $u^\dagger_{M_p}$ converges
  towards $u^\dagger_\infty$ in the $H^t(\Omega)$ norm, as $p\rightarrow \infty$. This also implies convergence in $C^2(\Omega)$ due to the continuous embedding of
  $H^t(\Omega)$ into $C^2(\Omega) \cap C(\overline{\Omega})$.}

  \yc{We now show that $u^\dagger_{\infty}$ satisfies the desired PDE in \eqref{elliptic-proto-PDE}. Denote by $v = -\Delta u^\dagger_{\infty}+ \tau\big(u^\dagger_{\infty}\big) -f \in C(\Omega)$ and $v_p = -\Delta u^\dagger_{M_p}+ \tau\big(u^\dagger_{M_p}\big) -f \in C(\Omega)$. For any $\bx \in \Omega$ and $p\geq 1$, use of the triangle inequality shows that
  \begin{equation}
  \begin{aligned}
  \label{eqn:triangle ineq}
    \left|v(\bx)\right|&\leq \min_{1\leq m \leq M_{p,\Omega}}\left(\left|v(\bx) -v(\bx_m)\right|+\left|v(\bx_m)-v_p(\bx_m)\right|\right)\\
    &\leq \min_{1\leq m \leq M_{p,\Omega}}\left|v(\bx) -v(\bx_m)\right| +\|v-v_p\|_{C(\Omega)} \, ,
  \end{aligned}
  \end{equation}
  where we have used the fact that $v_p(\bx_m)=0$, and $M_{p,\Omega}$ is the number of interior points associated with the total $M_p$ collocation points. Since $v$ is continuous in a compact domain, it is also uniformly continuous. Thus, it holds that $\lim_{p \to \infty} \min_{1\leq m \leq M_{p,\Omega}}\left|v(\bx) -v(\bx_m)\right| = 0$ since the fill-distance converges to zero by assumption. Moreover, we have that $v_p$ converges to $v$ in the $C(\Omega)$ norm due to the $C^2(\Omega)$ convergence from $u^\dagger_{M_p}$ to $u^\dagger_{\infty}$. Combining these facts with \eqref{eqn:triangle ineq}, and taking $p \to \infty$, we obtain $v(\bx)=0$, and thus $-\Delta u^\dagger_{\infty}(\bx)+ \tau\big(u^\dagger_{\infty}(\bx)\big) =f(\bx)$, for $\bx \in \Omega$. Following a similar argument, we get $u^\dagger_{\infty}(\bx)=g(\bx)$ for $\bx \in \partial \Omega$. Thus, $u^\dagger_{\infty}$ is a classical solution to \eqref{elliptic-proto-PDE}. By assumption, the solution is unique, so we must have
  $u^\dagger_\infty=u^\star$. As the limit $u^\dagger_\infty$ is independent of the particular subsequence, the whole sequence
  $u^\dagger_M$ must converge to $u^\star$ in $H^t(\Omega)$. Since $t \in (2+d/2,s)$, we also get pointwise convergence and convergence in $H^t(\Omega)$ for any $t <s$.
  }
  {}
  \end{proof}




We note that  this convergence theorem  requires $K$ to be adapted to the solution space of the PDE so that $u^\star$ belongs to $\mU$. In our numerical experiments, we will use a Gaussian kernel.
However, if $f$ or the boundary $\partial \Omega$ are not sufficiently regular, then
the embedding conditions $u^\star\in \mU \subseteq H^s(\Omega)$  may  be satisfied by using kernel as
the Green's function of some power of the Laplacian, instead of a Gaussian kernel; the latter induces smoothness
on $\mU$ which may be incompatible with the regularity of $u^\star$ for irregular $f$ and $\partial \Omega$.
 We will provide further insights into the convergence properties of our method, including an error bound between
 the solution  $u^\star$  and a minimizer $ u^\dagger$  in Section  \ref{sec:error-estimates}.

    \begin{figure}[htp]
     \centering
     \includegraphics[width=15cm]{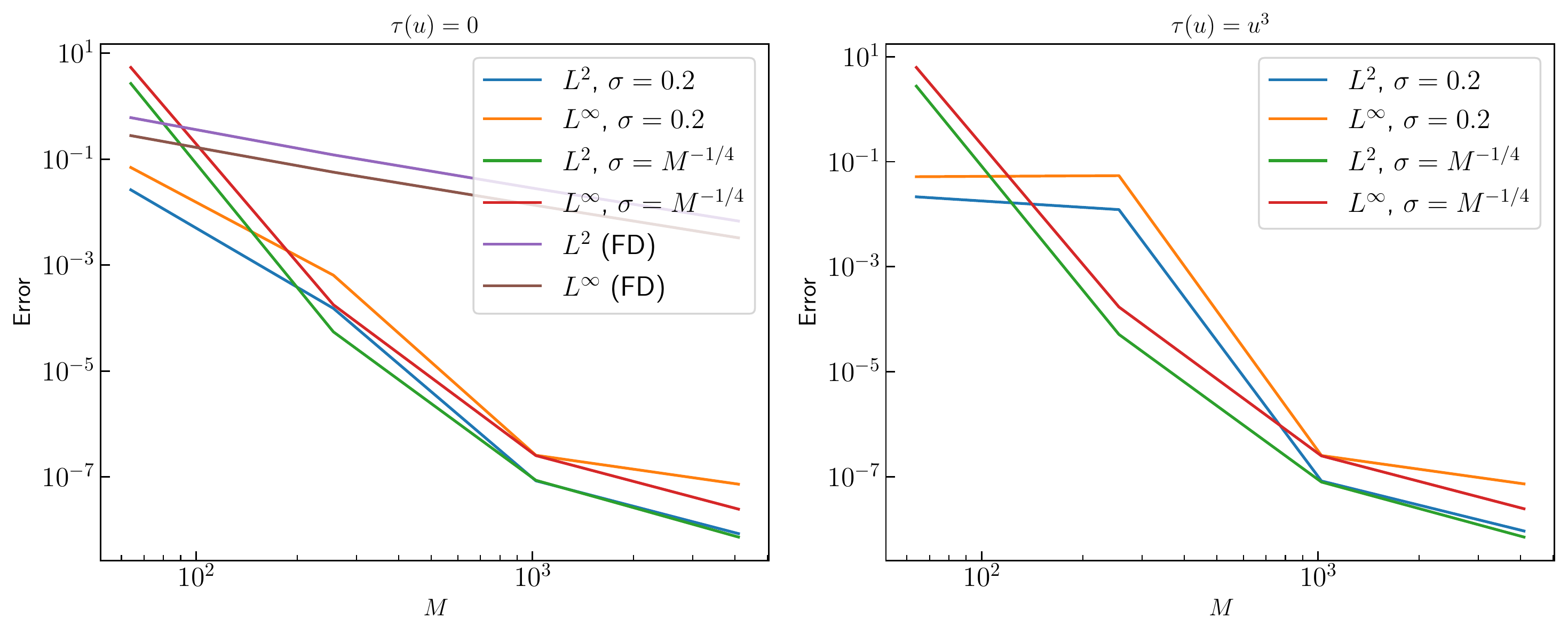}
     \caption{$L^2$ and $L^{\infty}$ error plots for numerical approximations of
     $u^\star$, the solution to \eqref{elliptic-proto-PDE},  as a function of
     the number of collocation points $M$. Left: $\tau(u) =0$; both the kernel collocation method using Gaussian kernel with $\sigma = 0.2$ and  $M^{-1/4}$ and the finite difference (FD) method were implemented. Right: $\tau(u)=u^3$; the kernel collocation method using Gaussian kernel with $\sigma = 0.2$ and  $M^{-1/4}$ were used. In both cases, an adaptive nugget term  with global parameter \bh{$\eta = 10^{-13}$ was used to regularize the kernel matrix $\Theta$ (see Appendix~\ref{sec:nuggets} for details)}.}
     \label{fig:elliptic convergence curve}
 \end{figure}

 \subsubsection{Numerical Framework}\label{sec:intro:implementation}
 The representation of $u^\dagger$ via the optimization problem \eqref{running-example-unconstrained-optimization-problem} is the cornerstone of
 our numerical framework for solving nonlinear PDEs. Indeed, efficient
 solution of \eqref{running-example-unconstrained-optimization-problem}, and in turn
 construction and inversion of the matrix $K(\bphi, \bphi)$, are the most costly
 steps of our numerical implementation. We summarize several main ingredients of our algorithm below:
 \begin{itemize}
     \item We propose an efficient variant of the \bh{Gauss--Newton}  algorithm in
     Section \ref{sec:GN-algorithm}. Although, in general, \eqref{running-example-unconstrained-optimization-problem} is a nonconvex problem, our algorithm typically converges in  \as{between $2$ and $10$ steps in all the experiments we have conducted.}

 \item \as{In practice we perturb $\cc(\bphi, \bphi)$ to improve its condition number, and hence the numerical stability
 of the algorithm, by adding a small diagonal matrix; this
 perturbation is adapted
 to the problem at hand, as outlined in Appendix~\ref{sec:nuggets}; the approach generalizes the idea of a ``nugget'' as widely used
 in GP regression.}

 \item To evaluate the cost function in \eqref{running-example-unconstrained-optimization-problem}, we pre-compute the Cholesky factorization of \as{the (perturbed) kernel
 matrix and store it for multiple uses. State-of-the-art linear solvers for dense kernel matrices} can be used for this step.
 \end{itemize}

As a demonstration, we present here a simple experiment to show the convergence of our method. We take $d=2$, $\Omega=(0,1)^2$ and  $\tau(u) = 0$ or $u^3$ together
  with homogeneous Dirichlet boundary conditions $g(\bx) =0$.
  The true solution is prescribed to be $u^\star(\bx) = \sin(\pi x_1)\sin(\pi x_2) + 4\sin(4\pi x_1)\sin(4\pi x_2)$ and the corresponding right hand side $f(\bx)$ is computed using the equation.

 We choose the Gaussian kernel
  $$K(\bx,\by; \sigma) = \exp\Bigl(-\frac{|\bx-\by|^2}{2\sigma^2}\Bigr)\, ,$$
  with lengthscale parameter $\sigma$; we will fix this parameter, but note that
  the framework is naturally adapted to learning of such hyperparameters.
  We set $M=64,256,1024,4096$, sampling the collocation points on a uniform grid  of points
  within $\Omega$. We apply our algorithm to solve the PDEs and compute the $L^2$ and $L^\infty$ errors to $u^\star$.
  In the case $\tau(u)=0$, which corresponds to a linear PDE, we also compare our algorithm with a second-order finite difference (FD) method.
  For the nonlinear case  $\tau(u) = u^3$, we observe that
  the \bh{Gauss--Newton}  algorithm only needs 2 steps to converge.
 The errors versus $M$ are depicted in Figure \ref{fig:elliptic convergence curve}.
 The following observations can be made from this figure:
\begin{itemize}
    \item In the linear case $\tau(u)=0$, where the corresponding optimization problem \eqref{running-example-unconstrained-optimization-problem} is convex, our algorithm
    outperforms the FD method in terms of accuracy and rate of convergence. This can be attributed to the fact that the true solution is very smooth, and the Gaussian kernel allows for efficient approximation.

    \item \bh{The choice of the kernel parameter $\sigma$ has a
     profound impact on the accuracy and rate of convergence of the algorithm, especially when $M$ is not very large.
    This
    implies the importance of choosing a ``good kernel'' that is
    adapted to the solution space of the PDE, and highlights  the importance of hyperparameter
    learning.}

    \item In the nonlinear setting, our algorithm demonstrates similar convergence behavior
     to the linear setting. Once again, an appropriate choice of $\sigma$ leads to
     significant gains in solution accuracy.

\end{itemize}

\subsection{Relevant Literature}\label{sec:relevant-lit}
Machine learning and statistical inference approaches to numerical approximation have attracted
a lot of attention in  recent years thanks to their remarkable  success in engineering applications.
Such approaches can be broadly divided into two categories: (1) GP/Kernel-based methods, and (2) ANN-based methods.

GPs and kernel-based methods have long been used in function approximation, regression, and machine learning \cite{williams1996gaussian}.
  As surveyed in \cite{owhadi2019statistical}:
\begin{quote}
 \emph{``[kernel  approaches can be traced back to] Poincar{\'e}'s course in Probability Theory  \cite{Poincare:1896} and to the pioneering investigations of Sul'din \cite{sul1959wiener}, Palasti and Renyi \cite{PalastiRenyi1956}, Sard
 \cite{Sard1963},  Kimeldorf and Wahba \cite{Kimeldorf70} (on the correspondence between Bayesian estimation and spline smoothing/interpolation) and Larkin \cite{larkin1972gaussian} (on the correspondence between Gaussian
 process regression and numerical approximation). Although their study initially
attracted little attention among numerical analysts \cite{larkin1972gaussian}, it was revived in Information Based
Complexity (IBC) \cite{Traub1988}, Bayesian numerical analysis \cite{diaconis1988bayesian}, and more recently  in Probabilistic Numerics \cite{Hennig2015} and Bayesian numerical homogenization \cite{Owhadi:2014}.''}
\end{quote}
For recent reviews of the extensive applications of GP/kernel methods see \cite{cockayne2019bayesian, owhadi2019statistical, swiler2020survey} and \cite[Chap.~20]{owhadi2019operator}. In particular, they have been introduced to motivate novel methods for solving ordinary differential equations (ODEs) \cite{skilling1992bayesian,calderhead2009accelerating,schober2014probabilistic} \as{and underlie the collocation approach advocated
for parameter identification in ODEs as described in
\cite{ramsay2007parameter}.}
For PDEs, the use of kernel methods can be traced back to meshless collocation methods with radial basis functions
\cite{zhang2000meshless, wendland2004scattered, schaback2006kernel}. Furthermore,
a recent systematic development towards solving PDEs was initiated in \cite{Owhadi:2014, owhadi2017multigrid, owhadi2017gamblets, owhadi2019operator} and has lead to the identification of fast solvers for elliptic PDEs and dense kernel matrices \cite{schafer2017compression, schafer2020sparse} with state-of-the-art complexity versus accuracy guarantees.
The effectiveness of these methods has been supported by
 well-developed theory \cite{owhadi2019operator} residing at the interfaces between numerical approximation \cite{wahba1990spline,wendland2004scattered}, optimal recovery \cite{micchelli1977survey}, information-based complexity \cite{Traub1988}, and GP regression \cite{bogachev1998gaussian}, building on the perspective introduced in \cite{diaconis1988bayesian, larkin1972gaussian, micchelli1977survey, traub1984average, wahba1990spline}.
In particular there is a one to one correspondence \cite{owhadi2019operator, schafer2017compression} between (1) the locality of basis functions (gamblets) obtained with kernel methods  and the \emph{screening effect} observed in kriging \cite{stein2002screening},  (2)  Schur complementation and conditioning Gaussian vectors, and (3) the approximate low-rank structure of inverse stiffness matrices obtained with kernel methods and  variational properties of Gaussian conditioning.
Furthermore, although the approach of modeling a deterministic function (the solution $u^\star$ of a PDE) as a sample from a Gaussian distribution may seem counterintuitive, it can be understood (in the linear setting \cite{owhadi2019operator}) as an optimal minimax strategy for recovering $u^\star$ from partial measurements. Indeed, as
 in Von Neumann's theory of games, optimal strategies are mixed (randomized) strategies and (using quadratic norms to define losses) GP regression (kriging) emerges as an optimal minimax approach to numerical approximation \cite{micchelli1977survey, owhadi2019operator}.

On the other hand,
 ANN methods  can be traced back to \cite{djeridane2006neural, lagaris1998artificial, lagaris2000neural, rico1993continuous, uchiyama1993solving}  and,
 \as{although developed for ODEs several decades ago
 \cite{cochocki1993neural,rico1993continuous}, with some of the work generalized
 to PDEs \cite{krischer1993model}, their systematic development for solving PDEs has been initiated only recently. This systematic
 development includes
  the Deep Ritz Method} \cite{weinan2018deep}, Physics Informed Neural Networks \cite{raissi2019physics} (PINNs), DGN \cite{sirignano2018dgm},
 and  \cite{han2018solving} which employs a probabilistic formulation of the nonlinear PDE via  the Feynman-Kac formula.
Although the basic idea is to replace unknown functions with ANNs and minimize some loss with respect to the parameters of the ANN to identify the solution, there are by now many variants of ANN approaches, and we refer to   \cite{Karniadakis2021} for a recent review of this increasingly popular topic at the interface between scientific computing and deep learning.
 While ANN approaches have been shown to be empirically successful on complex problems (e.g., machine learning physics \cite{raissi2020hidden}), they may fail on simple ones \cite{wang2020and, van2020optimally}
if the architecture of the ANN is not adapted to the underlying PDE \cite{wang2020eigenvector}. Moreover, the theory of ANNs is still in its infancy; most ANN-based PDE solvers lack theoretical guarantees and convergence analyses are often limited to restricted linear instances \cite{shin2020convergence, wang2020and}. Broadly speaking, in comparison with kernel methods, ANN methods have both limited theoretical foundations and unfavorable complexity vs accuracy estimates. We also remark that ANN methods are suitable
for the learning of the parametric dependence of solutions of differential equations
\cite{zhu2018bayesian, bhattacharya2020model, li2020fourier, li2020multipole, boncoraglio2021active};
however, GP and kernel methods may also be used in this context, and the random feature method
provides a natural approach to implementing such methods in high dimensions \cite{nelsen2020random}.

Regardless, the theory and computational framework of kernel methods may  naturally be extended to ANN methods to
 investigate\footnote{Beyond the infinite width neural tangent kernel regime \cite{jacot2018neural, wang2020and}.} such methods and possibly accelerate  them by viewing them as ridge regression with data-dependent kernels and following  \cite{schafer2017compression, schafer2020sparse}.  To this end,  ANN methods can be interpreted as kernel methods with data-dependent kernels  in two equivalent ways: (1) as solving PDEs in an RKHS space spanned by a feature map parameterized by the initial layers of the ANN that is learned from data; or, (2) as kernel-based methods with kernels that are parameterized by the inner layers of the network.
 For instance,  \cite{owhadi2020ideas} shows that
 Residual Neural Networks  \cite{he2016deep} (ResNets) are  ridge regressors with
 warped kernels \cite{sampson1992nonparametric,perrin1999modelling}.
 Since our  approach employs an explicit kernel, it allows us to learn that kernel directly, as discussed in Subsection \ref{sec:KernelFlow}.
Given the notorious difficulty of developing
 numerical methods for nonlinear PDEs \cite{tadmor2012review},  it is to some degree surprising that (as suggested by our framework) (A) this difficulty can universally be decomposed into three parts: (1) the  compression/inversion of dense kernel matrices, (2) the selection of the kernel, and (3) the minimization of the reduced finite-dimensional optimization problem \eqref{PDE-nonlinear-representer-optimization}, and
 (B) a significant portion of the resulting analysis can be reduced to that of linear regression   \cite{owhadi2019operator}.

 Beyond solving PDEs, ANN methods have also been used in data-driven discretizations, and discovery of PDEs that allow for the identification of the governing model \cite{raissi2019physics, long2018pde, bar2019learning}; this leads to applications in IPs. Our method, viewed as a GP conditioned on PDE constraints at collocation points, can be interpreted
 as solving an IP  with Bayesian inference and a Gaussian prior \cite{stuart2010inverse}. Indeed, if we relax the
 PDE constraints as in Subsection~\ref{sec:relaxed-constraints} then a minimizer $u^\dagger$ coincides
 with a MAP estimator of a posterior measure obtained by viewing the PDE constraints as nonlinear pointwise measurements
 of a field $u$ with a Gaussian prior $\mN(0, \C)$. Bayesian IPs with Gaussian priors have been studied
 extensively (see \cite{stuart-bayesian-lecture-notes, cotter-approximation, stuart2010inverse} and references therein). The standard approach for their solution is to discretize the problem using spectral projection or
 finite element methods, and compute posterior MAP estimators \cite{dashti2013map}
 or employ Markov chain Monte Carlo algorithms \cite{cotter2013mcmc} to simulate posterior samples. Our abstract
 representation theorem outlined in Section~\ref{sec:GMs-on-Banach-Spaces} completely characterizes the
 MAP estimator of posterior measures with Gaussian priors in settings where the forward map
 is written as the composition of a nonlinear map with bounded linear functionals acting on the parameter.
 Indeed, this is precisely the approach that we employ in Section~\ref{sec:Inverse-Problems} to solve IPs
 with PDE constraints. However, the main difference between our methodology and existing methods in
 the literature is that we pose the IP as that of recovering the solution of the PDE $u^\dagger$ simultaneously with learning
 the unknown PDE parameter with independent Gaussian priors on both unknowns.

\ho{
We now turn to motivation for the GP interpretation.
  The PDE solvers obtained here are deterministic and could be described from an entirely classical numerical approximation perspective. However we emphasize the GP interpretation for two reasons: (i) it is
  integral to the derivation of the methods; and (ii) it allows the numerical solution of the
  PDE to be integrated into a larger engineering pipeline and, in that context, the posterior
  distribution of the GP conditioned on the PDE at collocation points provides a measure
  of uncertainty quantification. Using the GP perspective as a pathway to the discovery of numerical methods, was the motivation for the work in \cite{Owhadi:2014,owhadi2017multigrid}.
Indeed, as discussed in \cite{Owhadi:2014}, while the discovery of numerical solvers for PDEs is usually based on a combination of insight and guesswork, this process can be facilitated to the point of being automated,
using this GP perspective.  For instance, for nonsmooth PDEs,  basis functions with near optimal accuracy/localization tradeoff and operator valued wavelets can be identified by conditioning physics informed Gaussian priors on localized linear measurements (e.g., local averages or  pointwise values). Physics informed Gaussian priors can, in turn, be identified by (a) filtering uninformed Gaussian priors through the inverse operator \cite{Owhadi:2014}, or (b) turning the process of computing fast with partial information into repeated zero-sum games with physics informed losses (whose optimal mixed strategies are physics informed Gaussian priors) \cite{owhadi2017multigrid}.
The paper \cite{raissi2018numerical} generalized (a) to time-dependent PDEs by filtering uninformed priors
through linearized PDEs obtained via time stepping.
The paper \cite{cockayne2016probabilistic} emphasized the probabilistic interpretation of numerical errors obtained from this Bayesian perspective.} \as{In particular \cite[Sec.~5.2]{cockayne2016probabilistic}
describes a method identical to the one considered here (and \cite{Owhadi:2014}) for linear problems; in the setting of semi-linear PDEs, it is suggested in \cite{cockayne2016probabilistic} that
latent variables could be employed to efficiently sample from posterior/conditional measures \ho{(see also \cite{owhadi2020ideas} where latent variables were employed to reduce nonlinear optimal recovery problems via two-level optimization as in \ref{running-example-two-level-optimization-problem}).} Although the
methodology proposed in our paper agrees with that in
\cite{cockayne2016probabilistic} for linear problems, the
methodology we propose appears to be
more general, and differs in the case of nonlinear problems to
which both approaches apply.}

\subsection{Outline of Article}\label{sec:outline}
The remainder of this article is organized as follows. We give an overview of the abstract theory of
\as{GPs on Banach spaces in Section~\ref{sec:overview}; we
establish notation, and summarize basic results and ideas that
are used throughout the remainder of the article.} Section~\ref{sec:PDEs} is dedicated to
the development of our numerical framework for solving nonlinear PDEs with kernel \as{methods; we outline our
assumptions on the PDE, present a general convergence theory, discuss our approach to implementation,} and present numerical experiments. In Section~\ref{sec:Inverse-Problems}
we extend our nonlinear PDE framework to IPs and discuss the implementation differences between the PDE and IP settings, followed by numerical experiments involving a benchmark  IP in subsurface flow.
Finally, we present additional discussions, results, and possible extensions of
our method in Section~\ref{sec:discussion}.
\as{Appendix~\ref{app:reg-theta} is devoted to the small
diagonal regularization of kernel matrices (``nugget'' term)
and outlines general principles as well as specifics for
the examples considered in this paper.}

\section{Conditioning GPs on Nonlinear Observations}\label{sec:overview}
 In this section, we outline the abstract theory  of
 RKHSs/GPs \bh{and their connection to} Banach spaces endowed with \as{quadratic norms; this forms the
 framework for the proposed methodology to solve} PDEs and IPs.
 We start by recalling some basic facts \bh{about
  GPs} in Subsection~\ref{sec:GMs-on-Banach-Spaces}. This is followed
 in Subsection~\ref{sec:rep-theorems-and-constraints} by
 general results pertaining to conditioning of GPs on linear and nonlinear observations, \as{leading to} a representer theorem that is the cornerstone of our numerical algorithms.
 Some of these results may be known to experts, but to the best of our knowledge, they are
 not presented in the existing literature with the coherent goal of solving nonlinear problems
 via the conditioning of GPs; hence
 this material may be of independent interest to the reader.

\subsection{GPs and Banach Spaces Endowed with a Quadratic Norm}\label{sec:GMs-on-Banach-Spaces}

Consider a separable Banach space $\mU$  and its dual $\mU^\ast$  with their duality
pairing denoted by $[ \cdot, \cdot ]$. We further assume that $\mU$ is endowed with a quadratic norm $\| \cdot \|$, i.e., there exists a linear bijection $\K: \mU^\ast \to \mU$ that is symmetric  ($[\K \phi,\varphi]=[\K \varphi,\phi]$), positive ($[\K \phi,\phi]>0$ for $\phi \ne 0$), and such
 that
\begin{equation*}
  \| u\|^2 = [ \K^{-1} u, u ], \qquad \forall u \in \mU\, .
\end{equation*}
The operator $\K$  endows  $\mU$ and $\mU^\ast$ with the following
inner products:
\begin{equation*}
  \begin{aligned}
    \langle u, v \rangle &:= [ \K^{-1} u, v ],  &&u,v \in \mU\, , \\
    \langle \phi, \varphi \rangle_{\ast} &:= [ \phi,  \K \varphi ],  &&\phi, \varphi \in \mU^\ast\, .
\end{aligned}
\end{equation*}
Note that the
$\langle \cdot, \cdot \rangle_\ast$ inner product defines a norm on $\mU^\ast$ that coincides
with the standard dual norm of $\mU^\ast$, i.e.,
\begin{equation*}
  \| \phi \|_{\ast} = \sup_{u \in \mU, u \neq 0} \frac{ [\phi, u ] }{\| u\| } = \sqrt{ [\phi, \K \phi] }\, .
\end{equation*}

As in \cite[Chap.~11]{owhadi2019operator},
although $\mU, \mU^\ast$ are also  Hilbert spaces under the quadratic norms $\|\cdot\|$ and $\|\cdot\|_{\ast}$, we will keep using the Banach space terminology to emphasize the fact that our dual pairings will not be based on the
inner product through the Riesz representation theorem,
  but on a  different  realization of the dual space.
  A particular case of the setting considered here is $\mU=H^s_0(\Omega)$  (writing $H^s_0(\Omega)$ for the closure of the set of smooth functions with compact support in $\Omega$ with respect to the Sobolev norm $\|\cdot\|_{H^s(\Omega)}$), with its dual
  $\mU^\ast=H^{-s}(\Omega)$ defined by the pairing   $[\phi,v]:=\int_{\Omega}{\phi u}$ obtained from the  Gelfand triple
 $H^s(\Omega) \subset L^2(\Omega) \subset H^{-s}(\Omega)$. Here $\K$
 can be defined as solution map
 of an elliptic operator\footnote{Given $s>0$, we call an  invertible operator $\mathcal{L}\,:\, H^s_0(\Omega)\rightarrow H^{-s}(\Omega)$ elliptic, if it is  positive and symmetric in the sense that $\int_\Omega u \mathcal{L} u\geq 0$ and $\int_\Omega u \mathcal{L} v= \int_\Omega v \mathcal{L} u\geq 0$.} mapping $H^s_0(\Omega)$ to $H^{-s}(\Omega)$.

We say that  $\xi$ is the {\it canonical GP} \cite[Chap.~17.6]{owhadi2019operator}
on $\mU$ and  write $\xi \sim \mN( 0, \K)$, if and only if
$\xi$ is a linear isometry from $\mU^\ast$ to a centered Gaussian space (a closed linear space of scalar valued centered Gaussian random variables). \bh{The word canonical indicates that the covariance operator of $\xi$ coincides with the
bijection $\mK$ defining the norm on $\mU$.}
Write $[\phi,\xi]$ for the image of $\phi\in \mU^\ast $ under $\xi$ and note that the following properties hold:
\begin{equation*}
  \EE [ \phi, \xi ] = 0 \quad \text{and} \quad
  \EE [ \phi, \xi ] [ \varphi, \xi ] = [ \phi, \K \varphi ],
  \quad \forall \phi, \varphi \in \mU^\ast\, .
\end{equation*}

\as{The space} $\mU$ coincides with the \bh{Cameron--Martin} space of the GP $\mN(0, \K)$.
In the setting where $\mK$ is defined through a covariance kernel $\cc$ (such as in
Subsection~\ref{sec:problem-setup} and later in Subsection~\ref{sec:GN-algorithm})
then $\mU$ coincides with the RKHS space of the kernel $\cc$ \cite[Sec.~2.3]{van2008reproducing}.

\subsection{Conditioning GPs with Linear and Nonlinear Observations}\label{sec:rep-theorems-and-constraints}
Let $\phi_1, \dots, \phi_N$ be $N$ non-trivial elements of $\mU^\ast$ and define
\begin{equation}\label{Phi-def}
  \bphi := \left( \phi_1, \dots, \phi_N \right) \in (\mU^\ast)^{\otimes N}\, .
  \end{equation}
  Now consider the canonical GP $\xi \sim \mN( 0, \C)$, then $[\bphi, \xi]$ is an $\R^N$-valued Gaussian vector
  and  $[\bphi, \xi] \sim \mN(0, \Theta)$ where
  \begin{equation}\label{Theta-matrix-def}
    \Theta \in \R^{N \times N}, \qquad \Theta_{i,n} = [ \phi_i, \C \phi_n ],\quad  1\leq i,n \leq N\, .
  \end{equation}

The following proposition characterizes the conditional
  distribution of GPs under these linear observations;
  to simplify the statement it is useful to write $\C (\bphi, \varphi)$ for the vector with entries
   $[ \phi_i , \C \varphi]$. This type of vectorized  notation is used in \cite{owhadi2019operator}.

  \begin{proposition}\label{prop:conditional-GP-with-linear-measurements}
    Consider a vector $\by \in \R^N$ and the canonical GP $\xi \sim \mN(0, \C)$.
    Then $\xi$  conditioned on $ [\bphi,\xi] = \by$
    is also Gaussian. Moreover if $\Theta$ is invertible then
    $\Law [ \xi| [\bphi,\xi] = \by ] = \mN( u^\dagger, \C_\bphi)$ with conditional mean defined by $u^\dagger = \by^T \Theta^{-1} \C \bphi$
    and conditional covariance operator
    defined     by
 $       [ \varphi , \C_\bphi \varphi]=
   [ \varphi , \C \varphi] -
    \C (\varphi,\bphi) \Theta^{-1}\C (\bphi, \varphi) , \forall \vphi \in \mU^\ast.$
  \end{proposition}

  Proposition~\ref{prop:conditional-GP-with-linear-measurements} gives a finite representation of the
  conditional mean of the GP constituting a representer theorem  \cite[Cor.~17.12]{owhadi2019operator}. Let us define the
  elements
  \begin{equation}
    \label{def:gamblets}
    \chi_i := \sum_{n=1}^N \Theta_{i,n}^{-1} \C \phi_n\, ,
  \end{equation}
  referred to as {\it gamblets} in the parlance of \cite{owhadi2017multigrid} which
  can equivalently  be characterized as
  the minimizers of the variational problem
  \begin{equation}\label{gamblet-optimization-definition}
    \left\{
      \begin{aligned}
        &\minimize_{\chi \in \mU} &&   \|  \chi \|  \\
        & \st  &&[ \phi_n, \chi ] = \delta_{i,n}, \qquad
        n = 1, \dots, N\, .
      \end{aligned}
      \right.
    \end{equation}
    This fact further enables the variational characterization of the conditional
    mean $u^\dagger$ directly in terms of the gamblets $\chi_n$.
    \begin{proposition}\label{prop:variational-conditional-mean-characterization-linear-case}
      Let $u^\dagger = \EE \left[ \xi | [\bphi,\xi] = \by \right]$ as in
      Proposition~\ref{prop:conditional-GP-with-linear-measurements}. Then $u^\dagger = \sum_{n=1}^N y_n \chi_n$
      is the unique minimizer of
      \begin{equation*}
    \left\{
      \begin{aligned}
        & \minimize_{u \in \mU} &&   \| u \| \\
        & \st && [ \phi_n, u ] = y_n, \qquad
        n = 1, \dots, N\, .
      \end{aligned}
      \right.
    \end{equation*}
  \end{proposition}

  Proposition~\ref{prop:variational-conditional-mean-characterization-linear-case} is the cornerstone
  of our methodology for solution of nonlinear PDEs. It is also useful for the
  solution of IPs. For this purpose consider
   nonlinear functions $G: \R^N \to \R^I $ and  $F: \R^N \to \R^{M}$  and  vectors $\bo \in \R^I$ and  $\by \in \R^M$ and consider the
  optimization  problem:
  \begin{equation}\label{generic-nonlinear-constrained-optimization-problem}
    \left\{
        \begin{aligned}
          & \minimize_{u \in \mU} && \|u\|^2 +  \frac{1}{ \gamma^2}  | G([\bphi, u]) - \bo |^2  \\
        & \st &&  F( [ \bphi, u ]) = \by\, ,
      \end{aligned}
      \right.
    \end{equation}
    where  $\gamma  \in \R$ is a parameter. We will use this formulation of IPs in PDEs,
    with $u$ concatenating the solution of the forward PDE problem and the unknown parameter; the
    nonlinear constraint on $F$ enforces the forward PDE and $G$ the observed noisy data.
    Then a representer theorem still holds for a minimizer of this problem
    stating that the solution has a finite expansion in terms of the gamblets $\chi_n$:

    \begin{proposition}\label{prop:nonlinear-representer-optimization}
      Suppose $(\bo, \by) \in \R^I \times  \R^M$ are fixed and $\Theta$ is invertible\footnote{Relaxing the interpolation constraints renders the invertibility assumption on $\Theta$ unnecessary.  Nevertheless we keep it for ease of presentation.}.
      Then $u^\dagger \in \mU$ is a minimizer of
      \eqref{generic-nonlinear-constrained-optimization-problem} if and only if
      $u^\dagger = \sum_{n=1}^N z^\dagger_n \chi_n$ and the vector
      $\bz^\dagger$ is a
      minimizer of
      \begin{equation}\label{nonlinear-representer-optimization}
         \left\{
        \begin{aligned}
          & \minimize_{\bz \in \R^N} &&  \bz^T \Theta^{-1} \bz
          + \frac{1}{ \gamma^2}  | G(\bz) - \bo |^2 \\
        & \st &&  F(\bz) = \by\, .
      \end{aligned}
      \right.
      \end{equation}
    \end{proposition}

    \begin{proof}
    The proof is nearly identical to the derivation of \eqref{eqn: finite dim optimization for elliptic eqn}  presented in  Section~\ref{sec:intro:representertheorem}. Simply observe that minimizing \eqref{generic-nonlinear-constrained-optimization-problem} is equivalent to minimizing
      \begin{equation}\label{jeduyye3}
    \minimize_{\bz \in \R^N\,:\, F( \bz) = \by} \left\{
        \begin{aligned}
          & \minimize_{u \in \mU} && \|u\|^2 +  \frac{1}{ \gamma^2}  | G(\bz) - \bo |^2   \\
        & \st &&  [ \bphi, u ] = \bz\, .
      \end{aligned}
      \right.
    \end{equation}
    Then solve the inner optimization problem for a fixed $\bz$ and
    apply Proposition~\ref{prop:conditional-GP-with-linear-measurements}.
    \end{proof}

We note that
this model assumes independent and identically distributed (i.i.d.) observation
    noise for the vector $\bo$ and can easily be extended to correlated observation noise by replacing
    the  misfit term $\frac{1}{\gamma^2} | G(\bz) - \bo |^2$ in  \eqref{generic-nonlinear-constrained-optimization-problem} with an appropriately weighted misfit term of the form $| \Sigma^{-1/2} (G(\bz) - \bo)|^2$ where $\Sigma$ denotes the
    covariance matrix of the observation noise.

 \begin{remark}\label{rem:Bayesian-interpretation}
 It is intuitive that a minimizer of the optimization problem we introduce
 and solve in this paper corresponds to a MAP point
 for the GP $\xi \sim \mN(0, \C)$ conditioned on PDE constraints at the collocation points.
 To prove this will require extension of
the approach introduced in \cite{dashti2013map}, for example,
and is left for future work. \bh{Here we describe this connection
informally in the absence of the equality constraints.
Consider the the prior measure $\mu_0 = \mN(0, \C)$ and consider the measurements
    $\bo = G([ \bphi, u]) + \eta,  \by = F([\bphi, u]) + \eta',
    \: \eta \sim \mN(0, \gamma^2 I), \eta' \sim \mN(0, \beta^2 I)$.
It then follows from Bayes' rule \cite{stuart2010inverse} that
the posterior measure of $u$ given the data $(\bo, \by$) is identified as the measure
\begin{equation*}
\begin{aligned}
     \frac{\dd \mu^{(\bo, \by)}}{\dd \mu_0}(u) & =
     \frac{1}{Z^{(\bo, \bu)}} \exp\left( -\frac{1}{2 \gamma^2} | G([\bphi, u])
    - \bo |^2 - \frac{1}{2 \beta^2} | F([\bphi, u]) - \by |^2 \right), \\
    Z^{(\bo, \by)} & :=\EE_{u \sim \mu_0} \exp\left( -\frac{1}{2 \gamma^2} | G([\bphi, u])
    - \bo |^2 - \frac{1}{2 \beta^2} | F([\bphi, u]) - \by |^2 \right).
\end{aligned}
\end{equation*}
The article \cite{dashti2013map} showed that the MAP estimators of $\mu^{(\bo, \by)}$
are solutions to
\begin{equation*}
        \begin{aligned}
          & \minimize_{u \in \mU} && \|u\|^2 +  \frac{1}{ \gamma^2}  | G([\bphi, u]) - \bo |^2   + \frac{1}{\beta^2} | F([\bphi, u]) - \by |^2.
      \end{aligned}
    \end{equation*}
Letting $\beta \to 0$ then yields \eqref{generic-nonlinear-constrained-optimization-problem}.
}

  \end{remark}

\section{Solving Nonlinear PDEs}\label{sec:PDEs}
In this section, we present our framework for solution of nonlinear PDEs by conditioning GPs on
nonlinear constraints.
In Subsection~\ref{sec:nonlinear-PDE-GP} we outline our abstract setting as well as our
assumptions on PDEs of interest; \as{this leads to} Corollary~\ref{cor:PDE-representer-theorem}
 which states an analogue of Proposition~\ref{prop:nonlinear-representer-optimization}
 in the PDE setting. We analyze the convergence of our method in Subsection~\ref{sec:Conv-theory}
 and discuss two strategies for dealing with the nonlinear PDE constraints in Subsection~\ref{sec:deadling-with-constraints}.
 Next, we present the details pertaining
 to numerical implementations of our method, including the choice of \bh{kernels and a \bh{Gauss--Newton} algorithm in Subsection~\ref{sec:implementation}.} Finally, we present a
 set of numerical experiments in Subsection~\ref{sec:numerics} that demonstrate the effectiveness of our
 method in the context of prototypical nonlinear PDEs.

\subsection{Problem Setup}\label{sec:nonlinear-PDE-GP}
Let us consider \bh{a bounded domain} $\Omega \subseteq \R^d$ for $d \ge 1$ and
 a nonlinear PDE of the form
\begin{equation}
  \label{prototypical-PDE}
  \left\{
  \begin{aligned}
    \mP(u^\star) (\bx) = f(\bx)
    ,& \qquad \forall \bx \in \Omega\, , \\
  \mB(u^\star) (\bx) = g(\bx),& \qquad \forall \bx \in \partial \Omega\, .
  \end{aligned}
  \right.
\end{equation}
Here $\mP$ is a nonlinear differential operator and $\mB$ is an appropriate boundary operator
with data $f,  g$.
Throughout this section and for brevity, we assume that the PDE at hand is well-defined pointwise and has a
unique strong solution; extension of our methodology to weak solutions is
 left as a future research direction. We then consider $\mU$ to be an appropriate
quadratic Banach space for the solution $u^\star$ such as a Sobolev space $H^s (\Omega)$
with sufficiently large regularity index $s >0 $.

We propose to solve the PDE \eqref{prototypical-PDE} by approximating $u^\star$ by a GP conditioned on satisfying
the PDE at a finite set of collocation points in $\overline{\Omega}$ and proceed to
approximate the solution by computing the MAP point of such a conditioned GP. More precisely, let $\{ \bx_i\}_{i=1}^M$
be a collection of points in $\overline{\Omega}$ ordered in such a way that
$\bx_1, \dots \bx_{\Md} \in \Omega$ are in the  interior  of $\Omega$
while $\bx_{\Md + 1}, \dots, \bx_{M} \in \partial \Omega$ are on the boundary.
Now let $\mU$ be a  quadratic Banach space with associated covariance operator
$\C: \mU^\ast \to \mU$
and consider the optimization problem:
\begin{equation}\label{PDE-optimization-form-initial}
    \left\{
      \begin{aligned}
        &\minimize_{u \in \mU} &&   \| u\|   \\
        & \st  &&  \mP(u)(\bx_m) = f(\bx_m),  &&\text{for } m = 1, \dots, M\, ,\\
        &  &&  \mB(u)(\bx_m) = g(\bx_m),  &&\text{for } m = \Md+1, \dots, M\, .\\
      \end{aligned}
      \right.
\end{equation}
In other words, we wish to approximate $u^\star$ with
\as{the minimum norm} element of the \bh{Cameron--Martin}
space of $\mN(0, \C)$ that satisfies the PDE and boundary data at the collocation points $\{\bx_i\}_{i=1}^M$.
\bh{In what follows we write $\mL(\mU; \mH)$ to denote the space of bounded and linear operators from $\mU$ to
another Banach space $\mH$.}
We  make the following assumption regarding the operators $\mP, \mB$:

\begin{assumption}\label{PDE-assumption}
There exist bounded and linear operators
$L_1, \dots, L_{Q} \in \mL( \mU; C( \Omega))$ in which $L_{1}, \dots, L_{Q_b} \in \mL( \mU; C( \partial \Omega))$ for some $1\leq Q_b \leq Q$, and \bh{there are
 maps $P: \R^{Q} \to \R$ and $B: \R^{Q_b} \to \R$, which may be nonlinear,} so
that $\mP(u)(\bx) $ and $\mB(u)(\bx)$  can be written as
 \begin{equation}
   \label{assumptions-on-mP-mB-PDE}
   \begin{aligned}
     \mP(u)(\bx) &= P \big( L_1(u)(\bx), \dots, L_{Q}(u)(\bx) \big), && \forall \bx \in \Omega\, , \\
     \mB(u)(\bx) &= B \big( L_{1 }(u)(\bx), \dots, L_{Q_b}(u)(\bx) \big), && \forall \bx \in \partial \Omega\, .
 \end{aligned}
\end{equation}
\end{assumption}

For prototypical nonlinear PDEs the   $L_q$ for $1\leq q \leq Q$  are linear differential operators such as first or second order
derivatives while the maps $P$ and $B$ are often simple algebraic nonlinearities.
Furthermore, observe that for ease of presentation
we are assuming
fewer linear operators are
used to define the boundary conditions than the operators that define the PDE in the interior.

\begin{runningexample}{NE}[Nonlinear Elliptic PDE]
  Recall the nonlinear elliptic PDE \eqref{elliptic-proto-PDE}
  and consider the linear operators and nonlinear maps
  \begin{equation*}
    L_1: u \mapsto u, \quad L_2: u \mapsto \Delta u, \quad  P(v_1, v_2) = - v_2 + \tau(v_1), \quad B(v_1) = v_1\, ,
    \end{equation*}
    where we took $Q = 2$ and $Q_b = 1$.
  Then this equation readily satisfies Assumption~\ref{PDE-assumption} whenever
  the solution is sufficiently regular so that $L_2(u)$ is well-defined pointwise within $\Omega$.
\end{runningexample}

Under Assumption~\ref{PDE-assumption} we can then
define the functionals $ \phi_{m}^{(q)} \in \mU^\ast$ by setting
\begin{equation}\label{phi-definition}
  \phi_{m}^{(q)} := \updelta_{\bx_m} \circ L_q, \quad
  \text{where} \quad
  \left\{
  \begin{aligned}
     &1 \le m \le M,  &&\text{if} \quad 1 \le q \leq Q_b, \\
     &1 \le m \le M_\Omega,  &&\text{if} \quad Q_{b+1} \le q \leq Q.
  \end{aligned}
  \right.
\end{equation}
We further use the shorthand notation $\bphi^{(q)}$  to denote the
vector of dual elements $\phi^{(q)}_m$ for a fixed index $q$.
Observe that $ \bphi^{(q)} \in ( \mU^\ast)^{\otimes M}$ if $q \leq Q_b$
while $ \bphi^{(q)} \in ( \mU^\ast)^{\otimes M_{\Omega}}$
if $q > Q_b$.
We further write
\begin{equation}
\label{eq:useful_to_highlight}
N= MQ_b+M_{\Omega}(Q-Q_b)
\end{equation}
and define
\begin{equation}\label{phi-PDE-def}
    \bphi := (  \bphi^{(1)}, \dots, \bphi^{(Q)} ) \in ( \mU^\ast)^{\otimes N}\, .
\end{equation}
To this end, we  define the measurement
vector $\by \in \R^M$ by setting
\begin{equation}\label{y-definition}
  y_m  = \left\{
    \begin{aligned}
      & f(\bx_m), && \text{if } m \in \{1, \dots, \Md\}\, , \\
      & g(\bx_m), && \text{if } m \in \{ \Md+1, \dots, M\}\, ,
    \end{aligned}
    \right.
  \end{equation}
  as well as the
nonlinear map
\begin{equation}\label{F-map-PDE-definition}
  \big( F([\bphi, u]) \big)_m :=  \left\{
    \begin{aligned}
      &P( [ \phi_{m}^{(1)}, u], \dots, [\phi_{m}^{(Q)}, u] )  &&\text{if } m \in \{1, \dots, \Md\}, \\
      &B( [\phi_{m}^{(1)}, u], \dots, [\phi_{m}^{(Q_b)}, u])  &&\text{if } m \in \{ \Md + 1, \dots, M\}.
    \end{aligned}
    \right.
  \end{equation}
We can now rewrite the optimization problem \eqref{PDE-optimization-form-initial}
in the same form as \eqref{generic-nonlinear-constrained-optimization-problem}:
\begin{equation}\label{PDE-optimization-form-standard}
    \left\{
      \begin{aligned}
        & \minimize_{u \in \mU} && \| u \|   \\
        & \st && F([ \bphi, u ]) = \by.
      \end{aligned}
      \right.
    \end{equation}
    Then  a direct appliction of Proposition~\ref{prop:nonlinear-representer-optimization}
    yields the
 following corollary.
 \begin{corollary}\label{cor:PDE-representer-theorem}
   Suppose Assumption~\ref{PDE-assumption} holds, $\C$ and $\Theta$ are invertible, and
   define $\bphi, F, \by$ as above. Then
   $u^\dagger$ is a minimizer of \eqref{PDE-optimization-form-initial} if and only if
   $u^\dagger = \sum_{n=1}^N z^\dagger_n \chi_n$ where the $\chi_n$ are the gamblets defined according to
   \eqref{gamblet-optimization-definition}
   and $\bz^\dagger$ is a
   minimizer of
   \begin{equation}\label{PDE-nonlinear-representer-optimization}
         \left\{
        \begin{aligned}
        & \minimize_{\bz \in \R^N} &&   \bz^T \Theta^{-1} \bz \\
        & \st &&  F  (\bz) = \by.
      \end{aligned}
      \right.
      \end{equation}
    \end{corollary}
 The above corollary is the foundation of our numerical algorithms for approximation of
 the solution $u^\dagger$, as $\Theta^{-1}$  and the gamblets $\chi_n$ can be  approximated offline
 while the coefficients $z^\dagger_n$ can be computed by solving the optimization problem
 \eqref{PDE-nonlinear-representer-optimization}.

To solve \eqref{PDE-nonlinear-representer-optimization} numerically, we will present two different approaches that transform it to an unconstrained optimization problem. Before moving to that in Subsection \ref{sec:deadling-with-constraints}, we discuss the convergence theory first in the next section.

\subsection{Convergence Theory}\label{sec:Conv-theory}
We state and prove a more general version of Theorem~\ref{thm:intro:convergence}
for our abstract setting of PDEs on Banach spaces with quadratic norms stating that a minimizer $u^\dagger$
of \eqref{PDE-optimization-form-initial} converges to the true solution $u^\star$ under sufficient
regularity assumptions and for appropriate choices of the operator $\mK$.

\begin{theorem}\label{thm:convergence}
  Consider the PDE \eqref{prototypical-PDE} and
  suppose that $\mU \subset \mH \subset C^{t}(\Omega)\cap C^{t'}(\overline{\Omega})$ where $\mH$ is
  a  Banach space such  that the
  first inclusion from the left is given by a compact embedding  and $t\ge t' \ge 0$ are
  sufficiently large so that \as{all derivatives
  appearing in the PDE are} defined pointwise for elements of $C^{t}(\Omega)\cap C^{t'}(\overline{\Omega})$.
  Furthermore
  assume that the PDE  has a unique classical solution $u^\star \in \mU$ and that, as $M\rightarrow \infty$,
  \begin{equation*}
  \sup_{\bx\in \Omega} \: \min_{1\leq m \leq M_\Omega}|\bx-\bx_m|\rightarrow 0
  \quad  \text{and} \quad
  \sup_{\bx\in \partial \Omega} \: \min_{M_\Omega + 1 \le m \le M}|\bx-\bx_m|\rightarrow 0.
  \end{equation*}
  Write $u^\dagger_M$ for a minimizer of \eqref{PDE-optimization-form-initial}
  with $M$ distinct collocation points. Then, as $M \rightarrow \infty$, the sequence
  of minimizers $u^\dagger_M$ converges towards $u^\star$ pointwise in $\Omega$ and in $\mH$.
\end{theorem}

\begin{proof}
  \yc{The method of proof is  similar to that of Theorem~\ref{thm:intro:convergence}. Indeed, by the same argument as in the first paragraph of the proof for Theorem~\ref{thm:intro:convergence}, there exists a subsequence $u_{M_p}$ that converges to $u_{\infty}^\dagger$ in $\mH$. This also implies convergence in $C^t(\Omega)$ and $C^{t'}( \overline{\Omega})$ \as{due to the assumed continuous embedding of
  $\mH$ into $C^{t}(\Omega)\cap C^{t'}(\overline{\Omega})$.} Since $t\ge t' \ge 0$ are
  sufficiently large so that all derivatives appearing in the PDE are defined pointwise for elements of $C^{t}(\Omega)\cap C^{t'}(\partial \Omega)$, we get that $\mP u_{M_p}$ converges to $\mP u^\dagger_{\infty}$ in $C(\Omega)$ and $\mP u^\dagger_{\infty} \in C(\Omega)$.
  As $\Omega$ is compact, $u^\dagger_{\infty}$ is also uniformly continuous in $\Omega$.

  For any $\bx \in \Omega$ and $p\geq 1$, the triangle inequality
  shows that
  \begin{equation}
      \begin{aligned}
         |\mP(u^\dagger_{\infty})(\bx)-f(\bx)| &\leq \min_{1\leq m \leq M_{p,\Omega}} \left(|\mP(u^\dagger_{\infty})(\bx) - \mP(u^\dagger)(\bx_m)| + |\mP(u^\dagger)(\bx_m)-\mP(u_{M_p})(\bx_m)| \right)\\
         & \leq \min_{1\leq m \leq M_{p,\Omega}} |\mP(u^\dagger_{\infty})(\bx) - \mP(u^\dagger_{\infty})(\bx_m)| + \|\mP u^\dagger_{\infty}-\mP u_{M_p}\|_{C(\Omega)}\, ,
      \end{aligned}
  \end{equation}
  where in the first inequality we have used the fact that $\mP(u_{M_p})(\bx_m) = f(\bx_m)$. Here $M_{p,\Omega}$ is the number of interior points associated with the total $M_p$ collocation points. Taking $p \to \infty$ and using the uniform continuity of $\mP u^\dagger_{\infty}$ and the $C(\Omega)$ convergence from $\mP u_{M_p}$ to $\mP u^\dagger_{\infty}$, we derive that $\mP(u^\dagger_{\infty})(\bx)=f(\bx)$. In a similar manner we can derive $\mB(u^\dagger_{\infty})(\bx) = g(\bx)$. Thus, the limit $u^\dagger_{\infty}$ is a classical solution to the PDE. By the uniqueness of the solution we must have $u^{\dagger}_{\infty}=u^\star$. Finally, as the limit $u^\dagger_\infty$ is independent of the
   choice of the subsequence, the whole sequence $u^\dagger_M$ must converge to $u^\star$ pointwise and in $\mH$.}
\end{proof}

We note that while this theorem does not provide a rate for convergence of $u^\dagger$ towards $u^\star$ it relies
on straightforward conditions that are readily verifiable for prototypical PDEs. Typically we
choose $t, t'> 0$ large enough so that  the PDE operators $\mP, \mB$ are pointwise defined for the elements of $C^{t}(\Omega)\cap C^{t'}(\overline{\Omega})$ (e.g., $t>$ order of PDE $+d/2$)
and take the space $\mH$ to be a Sobolev-type space of appropriate regularity for the inclusion $\mH \subset C^{t}(\Omega)\cap C^{t'}(\partial\Omega)$ to hold; also see the conditions of Theorem~\ref{thm:intro:convergence} and the subsequent discussion. The compact embedding $\mU \subset \mH$ can then be
ensured by an appropriate choice of the covariance operator $\mK$ (or the associated kernel $K$). However, this choice
should result in a sufficiently large space $\mU$ that includes the solution $u^\star$ of the PDE.
Our conditions on the collocation points $\{ \bx_m  \}_{m=1}^M$ simply ensure that these points form a dense subset of
$\overline{\Omega}$  as $M \to \infty$.

\subsection{Dealing with the Constraints}\label{sec:deadling-with-constraints}
Now, we turn our attention to the equality constraints in \eqref{PDE-nonlinear-representer-optimization}
and present two strategies for elimination or relaxation of these constraints; these transform the optimization problem to an unconstrained one. They are crucial
preliminary steps before introducing our numerical framework.

 \subsubsection{Eliminating the Equality Constraints}\label{sec:eliminating-constraints}

 The equality constraints in \eqref{PDE-nonlinear-representer-optimization} can be
  eliminated under slightly stronger assumptions on the maps $P, B$. In particular, suppose
  that the following assumption holds:
  \begin{assumption}
  \label{a:added}
  The equations
 \begin{equation*}
   P(v_1, \dots, v_{Q}) = y, \qquad B(v_{1}, \dots, v_{Q_b}) = y\, ,
 \end{equation*}
 \bh{can be solved as \as{finite-dimensional algebraic equations}, i.e., there exist  $\hP: \R^{Q-1} \to \R$ and $\hB:\R^{Q_b -1} \to \R$ so
 that
 \begin{equation}\label{nonlinearity-solution-maps}
   v_j = \hP( v_1, \dots, v_{j-1}, v_{j+1},  \dots,  v_{Q}, y), \qquad v_{k} = \hB(v_1, \dots, v_{k-1}, v_{k+1}, \dots,  v_{Q_b}, y)\, ,
 \end{equation}
  for selected indices $j \in \{1, \dots, Q\}$ and $k \in \{1, \dots, Q_b\}$.}
\as{Then for integer $N$ defined by \eqref{eq:useful_to_highlight},
and using the solution maps $\hP, \hB$, we can then define a new solution map}
 $\hF: \R^{N - M} \times \R^M \to \R^N$ so that
 \begin{equation*}
   F (\bz) = \by \quad  \text{ if and only if } \quad
   \bz = \hF(\bw, \by), \quad \text{ for a unique } \bw \in \R^{N - M}\, .
\end{equation*}
\end{assumption}
With this new solution map we can rewrite \eqref{PDE-nonlinear-representer-optimization}
as an unconstrained optimization problem.

\begin{corollary}\label{cor:PDE-unconstrained-representer-theorem}
\bh{Let Assumption \ref{a:added}  and the conditions of Corollary~\ref{cor:PDE-representer-theorem} hold.} Then
  $u^\dagger$ is a minimizer of \eqref{PDE-optimization-form-initial} if and only if
  $u^\dagger = \sum_{n=1}^N z^\dagger_n \chi_n$  with $\bz^\dagger = F'(\bw^\dagger, \by)$
  and $\bw^\dagger \in \R^{N- M}$ is a minimizer of
\begin{equation}\label{PDE-unconstrained-representer-theorem}
  \minimize_{\bw \in \R^{N - M}} \quad  \hF(\bw, \by)^T \Theta^{-1} \hF(\bw, \by)\, .
\end{equation}
\end{corollary}

\begin{runningexample}{NE}
Let us recall that we already eliminated the
equality constraints in the context of the PDE \eqref{elliptic-proto-PDE} through the calculations leading
to the unconstrained minimization problem  \eqref{running-example-unconstrained-optimization-problem}. In that example, we used the calculation
\begin{equation*}
    P(v_1, v_2) = -v_2 + \tau(v_1) = y \Leftrightarrow v_2 = \tau(v_1) - y = \hP(v_1, y)\, ,
\end{equation*}
that is, we solved $\Delta u$ in terms of $\tau(u)$ and the source term in the interior of the domain in
order to eliminate the PDE constraint. We further imposed the boundary conditions exactly since
the boundary map $B$ is simply the pointwise evaluation function in that example.

\bh{Alternatively, we could eliminate  $v_1$ by setting $v_1 = \tau^{-1}( y + v_2)$, assuming that $\tau^{-1}$ has closed form. While both elimination strategies are conceptually valid they may lead to
very different optimization problems. The former corresponds to solving for the values of $u$ at the
collocation points while the latter solves for the values of $\Delta u$ at the interior points under
Dirichlet boundary conditions at the boundary collocation points.
}
\end{runningexample}
\subsubsection{Relaxing the Equality Constraints}\label{sec:relaxed-constraints}
The choice of the solution maps $\hP, \hB$ in \eqref{nonlinearity-solution-maps}, i.e., the choice of
the variable which the equations are solved for, has an impact
on the conditioning of \eqref{PDE-unconstrained-representer-theorem}; it is  not a priori clear
that poor conditioning can always be avoided by choice of variables to solve for. Moreover, for certain
nonlinear PDEs \bh{Assumption~\ref{a:added} may not hold}. In such cases it may be useful to
relax the equality constraints in
\eqref{PDE-optimization-form-standard} and instead consider
a loss of the following form:
\begin{equation}\label{PDE-relaxed-optimization}
         \left.
        \begin{aligned}
          & \minimize_{u \in \mU} \quad \|u\|^2
          + \frac{1}{\beta^2}  | F ([\bphi,u]) - \by |^2\, ,
      \end{aligned}
      \right.
      \end{equation}
      where $\beta^2 >0$ is a small positive parameter. Likewise
      \eqref{PDE-nonlinear-representer-optimization}
      can be relaxed to obtain
\begin{equation}\label{PDE-relaxed-representer-optimization}
         \left.
        \begin{aligned}
          & \minimize_{\bz \in \R^N} \quad  \bz^T \Theta^{-1} \bz
          + \frac{1}{\beta^2}  | F (\bz) - \by |^2\, .
      \end{aligned}
      \right.
      \end{equation}
      Then a similar argument to the proof of Theorem~\ref{thm:convergence} can be used to show that
      a minimizer of the relaxed optimization problem for $u$ converges to the
      solution $u^\star$ of the PDE as the number of collocation points $M$ increases and
      the parameter $\beta$ vanishes.

      \begin{proposition}\label{thm:relaxed-constraint-convergence}
      Fix $\beta>0.$ Then the optimization problem \eqref{PDE-relaxed-optimization} has minimizer
        $u^\dagger_{\beta,M}$ which (assuming $\Theta$ to be invertible) may be expressed in the form
        $$u^\dagger_{\beta,M}:= \sum_{n=1}^N  z^\dagger_{\beta, n} \chi_n \in \mU\, ,$$ where $\bz^\dagger_\beta$
        denotes a minimizer of \eqref{PDE-relaxed-representer-optimization}.
        Under the assumptions of Theorem \ref{thm:convergence} it follows that,
         as $(\beta,M^{-1}) \to 0$,  the relaxed estimator
        $u^\dagger_{\beta,M}$  converges to  $u^\star$ pointwise and in $\mH$.
      \end{proposition}
\begin{proof}
By the arguments used in  Proposition \ref{prop:nonlinear-representer-optimization} the
minimizer of \eqref{PDE-relaxed-representer-optimization} delivers a minimizer of
\eqref{PDE-relaxed-optimization} in the desired form.
Since $u^\star$ satisfies $F ([\bphi,u^\star]) - \by=0$ we must have
$\|u^\dagger_{\beta,M}\| \leq \|u^\star\|$.
Then a compactness argument similar to that used in  the proof  of Theorem \ref{thm:convergence}, noting
that taking $\beta \to 0$ as $M \to \infty$ delivers exact satisfaction of the constraints in the
limit, yields the
desired result.
\end{proof}

      \begin{runningexample}{NE}
    When only part of the constraints $F(\bz)=\by$ can be explicitly solved, as is often the case for boundary values, we can also combine the elimination and relaxation approach.
        Employing the relaxation approach for the interior constraint and the elimination approach for the boundary constraints in \eqref{elliptic-proto-PDE}  amounts to replacing the
        optimization problem \eqref{eqn: finite dim optimization for elliptic eqn}, which is the analogue of
        \eqref{PDE-nonlinear-representer-optimization} for our running example, with the following problem
        for a small parameter $\beta^2 >0$:
               \begin{equation}
               \label{eqn: nonlinear-elliptic-PDE-relaxation-approach}
   \left\{
   \begin{aligned}
   \minimize_{\bz \in \R^{M + \Md}}  \quad   & \bz^T  \cc(\bphi,\bphi)^{-1}\bz  + \frac{1}{ \beta^2}
  \left[  \sum_{m=1}^{\Md} \left| z^{(2)}_m +\tau(z^{(1)}_m)- f(\bx_m) \right|^2
  \right] \\
  \text{s.t.} \quad & z_m^{(1)}=g(\bx_m), \quad \text{for } m = M_{\Omega}+1,..., M\, .
\end{aligned}
\right.
  \end{equation}
  We will numerically compare the above approach with the full elimination approach \eqref{running-example-unconstrained-optimization-problem} in Subsection \ref{sec:elliptic-PDE}.
      \end{runningexample}

\subsection{Implementation}\label{sec:implementation}
We now outline the details of a numerical algorithm for solution of
nonlinear PDEs based on the discussions of the previous subsection and in
particular Corollary~\ref{cor:PDE-representer-theorem}.
We discuss the construction of the matrix $\Theta$ in Subsection~\ref{sec:constructing-theta}
followed by a variant of the \bh{Gauss--Newton} algorithm in Subsection~\ref{sec:GN-algorithm} for solving the unconstrained or relaxed problems outlined in Subsections~\ref{sec:deadling-with-constraints}.
 \bh{We also note that strategies  for regularizing the matrix $\Theta$ by adding \as{small diagonal} (``nugget'') terms
 are collected  in Appendix~\ref{app:reg-theta}}.

\subsubsection{Constructing $\Theta$}\label{sec:constructing-theta}
We established through Corollary~\ref{cor:PDE-representer-theorem} that a solution to
\eqref{PDE-optimization-form-initial} can be completely identified by
$\bz^\dagger$ a minimizer of \eqref{PDE-nonlinear-representer-optimization}
as well as the gamblets $\chi_n$. Since here we are concerned with the strong form of
the PDE \eqref{prototypical-PDE} it is reasonable to assume that at the very least
$\mU \subset  C(\overline{\Omega})$; although we often require higher regularity
so that the PDE constraints can be imposed pointwise.
 This assumption suggests that our GP model for $u^\star$ can equivalently be
identified via a covariance kernel function as opposed to the covariance operator $\C$.
To this end, given a covariance operator $\C$ define the covariance kernel $\cc$ (equivalently Green's function of $\C^{-1}$) as
\begin{equation}
  \label{covariance-kernel-def}
  \cc: \overline{\Omega} \times \overline{\Omega} \mapsto \R, \qquad \cc( \bx, \bx') :=
  [ \updelta_{\bx}, \C \updelta_{\bx'} ]\, .
\end{equation}
It is known that the kernel $\cc$ completely characterizes the GP $\mN(0, \C)$ under mild conditions \cite{van2008reproducing};
that is $\mN(0, \C) \equiv \mathcal{GP}(0,\cc).$  Let us now consider the matrix $\Theta$ in block form
\begin{equation*}
  \Theta =
  \begin{bmatrix}
    \Theta^{(1,1)} & \Theta^{(1,2)} &  \cdots & \Theta^{(1, Q)} \\
    \Theta^{(2,1)} & \Theta^{(2,2)} & \cdots & \Theta^{(2,Q)} \\
    \vdots & \vdots & \ddots & \vdots \\
    \Theta^{(Q,1)} & \Theta^{(Q,2)} & \cdots & \Theta^{(Q,Q)}
  \end{bmatrix}\, .
\end{equation*}
Using the $L^2(\Omega)$ duality pairing between
$\mU$ and $\mU^\ast$
we can identify  the blocks
\begin{equation*}
  \Theta^{(q, j)} = \cc( \bphi^{(q)}, \bphi^{(j)})\, ,
\end{equation*}
where we used the shorthand notation of Subsection~\ref{sec:problem-setup}  for the kernel matrix,
with the $\bphi^{(q)}$ defined as in
\eqref{phi-definition} and the subsequent discussion. To this end the entries of
the $\Theta^{(q, j)}$ take the form
\begin{equation*}
  \Theta^{(q, j)}_{m, i} =  L_q^\bx L_j^{\bx'} \cc( \bx, \bx') \big|_{(\bx, \bx') = (\bx_m, \bx_i)}\, ,
\end{equation*}
where we used the superscripts $\bx, \bx'$ to denote the variables with respect to
which the differential operators $L_q, L_j$ act. Note that $\Theta \in \R^{N \times N}$
with $N = MQ_b+M_{\Omega}(Q-Q_b)$ following the
 definition of $ \bphi^{(q)}$ in Subsection \ref{sec:nonlinear-PDE-GP}.

\subsubsection{A \bh{Gauss--Newton} Algorithm}\label{sec:GN-algorithm}
Here we outline a variant of the \bh{Gauss--Newton
algorithm \cite[Sec.~10.3]{nocedal2006numerical}}
for solution of the unconstrained optimization problem
\eqref{PDE-unconstrained-representer-theorem}. Recall our definition of the maps $\hP, \hB$
in \eqref{nonlinearity-solution-maps} and in turn the map $\hF$. We then propose to
approximate a minimizer $\bw^\dagger$ of  \eqref{PDE-unconstrained-representer-theorem}
with a sequence of elements $\bw^{\ell}$ defined iteratively via
  $\bw^{\ell + 1} = \bw^{\ell} + \alpha^\ell  \delta \bw^{\ell},$
where $\alpha^\ell > 0 $ is an appropriate step size while
$\delta \bw^{\ell}$ is the minimizer of the optimization problem
\begin{equation*}\label{GN-alg-optimization-unconstrained}
  \minimize_{ \delta \bw \in \R^{N - M}} \quad
  \left( \hF(\bw^\ell, \by) +  \delta\bw^T  \nabla  \hF(\bw^\ell, \by) \right) ^T
  \Theta^{-1} \left( \hF(\bw^\ell, \by) +  \delta\bw^T  \nabla  \hF(\bw^\ell, \by) \right)\, ,
\end{equation*}
and the gradient of $\hF$ is computed with respect to the $\bw$ variable only.
\footnote{\bh{Note that our proposed method is nothing more than the standard Gauss--Newton
algorithm with Euclidean norm $|\cdot |$ defining the least-squares functional replaced with the weighted norm $| \Theta^{-1/2} \cdot |$ \cite[Sec.~10.3]{nocedal2006numerical}.} }

This approach can be applied also to solve the relaxed problem \eqref{thm:relaxed-constraint-convergence}
where this time we consider the sequence of approximations
  $\bz^{\ell + 1} = \bz^{\ell} + \alpha^\ell \delta \bz^\ell,$
where  $\delta \bz^\ell$ is  the minimizer of
\begin{equation*}\label{GN-alg-optimization-relaxed}
  \minimize_{ \delta \bz \in \R^{N}} \quad
  \left(  \bz^\ell + \delta \bz \right)^T \Theta^{-1} \left( \bz^\ell + \delta \bz\right)
  + \frac{1}{\beta^2}  \left| F(\bz^\ell) + \delta \bz^T \nabla F(\bz^\ell) - \by \right|^2\, .
\end{equation*}
Since   \eqref{GN-alg-optimization-unconstrained} and \eqref{GN-alg-optimization-relaxed}
are both  quadratic in $\delta \bw$ and $\delta \bz$ respectively,  they
can be solved exactly and efficiently at each step
and the step-size parameters $\alpha^\ell$ can be fixed or computed adaptively
using standard step-size selection techniques \cite{nocedal2006numerical}. However, in our  experiments in
Section~\ref{sec:numerics}, we  find that both algorithms converge quickly simply by setting
 $\alpha^\ell = 1$.

\begin{runningexample}{NE}
Let us return once more to the nonlinear elliptic PDE considered in Subsection~\ref{sec:problem-setup}.
Observe that \eqref{running-example-unconstrained-optimization-problem} is precisely in the form of \eqref{PDE-unconstrained-representer-theorem} and so in order to formulate our \bh{Gauss--Newton}
iterations we need to linearize the vector valued function
\begin{equation*}
\bw \mapsto
     \big(\bw, g(\bx_{\partial\Omega}), f(\bx_{\Omega})-\tau(\bw)\big)\, ,
\end{equation*}
which can easily be achieved by linearizing $\tau$. To this end, we solve \eqref{PDE-unconstrained-representer-theorem}
via the iteration $\bw^{\ell + 1} = \bw^\ell + \alpha^\ell \delta \bw^\ell$ where $\delta \bw^\ell $ is the minimizer
of the functional
\begin{equation*}
     \big(\bw^\ell + \delta \bw , g(\bx_{\partial\Omega}), f(\bx_{\Omega})-\tau(\bw^\ell) - \delta \bw^T \nabla \tau(\bw^\ell) \big)
      \cc(\bphi,\bphi)^{-1}
      \begin{pmatrix}\bw^\ell + \delta \bw \\
       g(\bx_{\partial\Omega}) \\
       f(\bx_{\Omega})-\tau(\bw^\ell) - \delta \bw^T \nabla \tau(\bw^\ell) \end{pmatrix}\, .
\end{equation*}
We also note that the sequence of approximations  obtained by the above
implementation of \bh{Gauss--Newton} coincides with  successive
kernel collocation approximations  of the solution of the following  particular  linearization of the
PDE,
\begin{equation}
-\Delta u + u \tau'(u^{n})=f-\tau(u^n) +u^n\tau'(u^{n})\, ,
\end{equation}
subject to the Dirichlet boundary conditions.
\end{runningexample}

\subsubsection{Computational Bottlenecks}\label{subsec:numerical-cost}
\bh{
The primary computational cost of our method lies in the approximation of the matrix $\Theta^{-1}$. 
Efficient factorizations and approximations of
$\Theta^{-1}$ have been studied extensively in the GP regression literature \cite{quinonero2005unifying} as well as spatial statistics, Kriging and numerical analysis
(see \cite{schafer2020sparse, schafer2017compression} and the discussions within).
In this paper, we do not employ these algorithms and choose instead to use standard $\mathcal{O}(N^3)$ algorithms
to factorize $\Theta$.

 The algorithm introduced in \cite{schafer2020sparse} is particularly interesting as
 it  directly approximates the Cholesky factors of $\Theta^{-1}$ by querying a subset of the entries of $\Theta$.
In fact, that algorithm alleviates the need for a \as{small
diagonal regularization (``nugget'')} term by directly computing the Cholesky factors of $\Theta^{-1}$  from the entries of $\Theta$. This could be done by extending the algorithm introduced and analyzed in \cite{schafer2020sparse}. This algorithm is based on the identification of an explicit formula for computing    approximate Cholesky factors $L$ minimizing the Kullback-Leibler divergence between $\mathcal{N}(0,\Theta^{-1})$ and $\mathcal{N}(0,LL^T)$ given a sparsity constraint on the entries of $L$.
The proposed formula  is  equivalent to the Vecchia approximation \cite{vecchia1988estimation} (popular in geostatistics).
The resulting algorithm outlined in
 \cite{schafer2020sparse} computes $\epsilon$ approximate Cholesky factors of $\Theta^{-1}$ in  $\mathcal{O}(N \log^{2d} (N/\epsilon))$ complexity by accessing $\mathcal{O}(N \log^{d} (N/\epsilon))$ entries of $\Theta$.

 Another possible bottleneck is the computation of the gamblets $\chi_n$. The articles \cite{owhadi2019operator, schafer2017compression} show that the  gamblets can be approximated with compactly supported functions in complexity  $\mathcal{O}(N \log^{2d+1} (N/\epsilon))$.
 We also note that  the complexity-vs-accuracy guarantees of \cite{owhadi2019operator, schafer2020sparse, schafer2017compression}  have only been established for   functionals $\phi_n$ that are  Dirac delta functions and
kernels  $\cc$ that are the Green's functions of arbitrary elliptic differential operators (mapping $H^s(\Omega)$ to $H^{-s}(\Omega)$).
Extension of those results to functionals $\phi_n$  considered here is an interesting future direction.
}

\subsection{Numerical Experiments for Nonlinear PDEs}\label{sec:numerics}
In this subsection, we implement our algorithm to solve several nonlinear PDEs, including the nonlinear elliptic equation in Subsection \ref{sec:elliptic-PDE}, Burgers' equation in Subsection \ref{sec:Burgers-Revisited} and the regularized Eikonal equation in Subsection \ref{sec:eikonal-PDE}. For all of these equations, we will start with a fixed $M$ and demonstrate the performance of our algorithm by showing the pattern of collocation points, the loss function history of the \bh{Gauss--Newton} iteration, and contours of the solution errors.  Then, we vary the value of $M$ and study how the errors change with respect to $M$.  We also compare the elimination and relaxation approaches for dealing with the nonlinear constraints.

All the experiments are conducted using Python with the JAX package
for automatic differentiation\footnote{\bh{We use JAX for convenience and
all derivatives in our methodology can be computed using standard techniques such as symbolic
computation or adjoint methods.}}. \as{In particular, we use} automatic differentiation
to form the kernel matrix $\Theta$ that involves derivatives of the kernel function, and to optimize the loss function via the \bh{Gauss--Newton} method.
\as{Details on the choice of small diagonal regularization
(``nugget'') terms for these experiments are presented
in Appendices~\ref{app:Nonlinear-Elliptic-Nugget} through \ref{app:Eikonal-Nugget}.}

\yc{\begin{remark}
\bh{In all of the numerical experiments in this section we used a
set of collocation points that are drawn randomly from the uniform distribution
over the domain $\Omega$,
as opposed to the deterministic uniform grid used} in Subsection \ref{sec:intro:implementation}. \as{The choice of the random collocation
points was made to highlight the flexibility of our methodology. Furthermore,
random collocation points are often used in other machine learning algorithms
for solution of PDEs such as PINNs \cite{raissi2019physics}
and so adopting this approach allows direct comparison with
such methods. We observed}
empirically that the random grids had similar accuracy to the deterministic uniform
grid in all experiments except for Burgers' equation in Subsection~\ref{sec:Burgers-Revisited}, where random collocation points outperformed
the uniform grid.
\bh{Understanding this surprising performance gap is an interesting problem related to
active learning and the acquisition of collocation points; \as{we do not address this issue
here}.}
\end{remark}
}

\begin{remark}
\as{Float64 data type was employed in the experiments below. This  allows the use of small diagonal
regularization (``nugget'') terms
(see Appendix~\ref{app:reg-theta} for details) which
do not affect accuracy in the computations described in this
paper. In contrast, if Float32 data type (the default setting in JAX) is used, we found the need to regularize $\Theta$ with
larger diagonal terms, leading to an observable accuracy floor.}
\end{remark}

\subsubsection{A Nonlinear Elliptic PDE}\label{sec:elliptic-PDE}
We revisit again the nonlinear elliptic equation in \eqref{elliptic-proto-PDE}. As in Subsection \ref{sec:intro:implementation}, we take $d=2$, $\Omega=(0,1)^2$ and  $\tau(u) = u^3$ together
  with homogeneous Dirichlet boundary conditions $g(\bx) =0$.
  The true solution is prescribed to be $u^\star(\bx) = \sin(\pi x_1)\sin(\pi x_2) + 4\sin(4\pi x_1)\sin(4\pi x_2)$ and the corresponding right hand side $f(\bx)$ is computed using the equation. We choose the Gaussian kernel
  $K(\bx,\by; \sigma) = \exp(-\frac{|\bx-\by|^2}{2\sigma^2})$ with a lengthscale parameter $\sigma$.
\begin{figure}[ht]
    \centering
    \begin{overpic}[width=15cm]{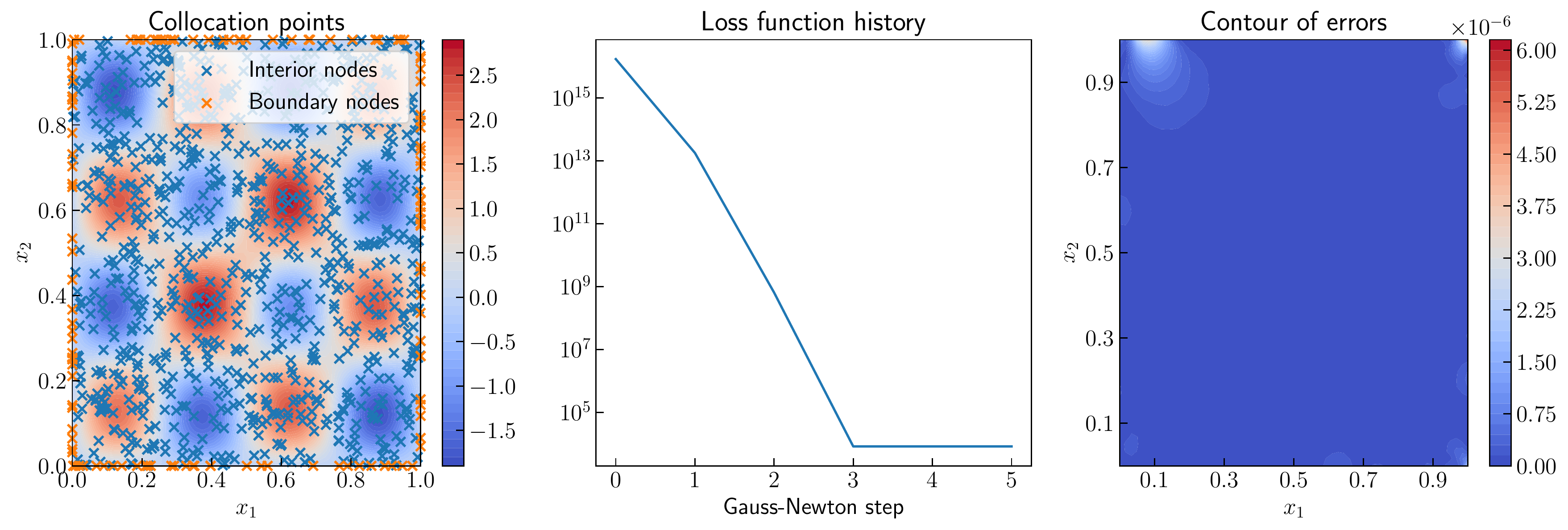}
    \put(14,-1){\scriptsize (a)}
    \put(50,-1){\scriptsize (b)}
    \put(82,-1){\scriptsize (c)}
    \end{overpic}
    \caption{Numerical results for the nonlinear elliptic PDE \eqref{elliptic-proto-PDE}: (a) a sample of collocation points and  contours of the true solution;
    (b) convergence history of the \bh{Gauss--Newton} algorithm;
    (c) contours of the solution error. \bh{An adaptive nugget term
    with global parameter $\eta = 10^{-13}$ was employed (see Appendix~\ref{app:Nonlinear-Elliptic-Nugget})}}
    \label{fig:nonlinear elliptic, demon}
\end{figure}

First, for $M=1024$ and $M_\Omega = 900$, we randomly sample collocation points in $\Omega$ as shown in  Figure~\ref{fig:nonlinear elliptic, demon}(a).
We further show an instance of the convergence history of the \bh{Gauss--Newton}  algorithm
in Figure~\ref{fig:nonlinear elliptic, demon}(b)
where we
 solved the unconstrained optimization problem \eqref{running-example-unconstrained-optimization-problem} after eliminating the
 equality constraints. We used kernel parameter $\sigma = M^{-1/4}$ \bh{and  appropriate
 nugget terms as outlined in  Appendix~\ref{app:Nonlinear-Elliptic-Nugget}}.
 We initiated the algorithm with a Gaussian random initial guess. We observe that only $3$ steps sufficed for  convergence. In Figure~\ref{fig:nonlinear elliptic, demon}(c), we show the contours of the solution error.  The
 error in the approximate solution is seen to be fairly uniform spatially, with larger errors
 observed near the boundary, when $M=1024.$ We
 note that the main difference between these experiments and those  in
 Subsection~\ref{sec:intro:implementation} is that here we used randomly distributed
 collocation points while a uniform grid  was used previously.

Next, we compare two approaches for dealing with the PDE  constraints
as outlined in Subsection~\ref{sec:deadling-with-constraints}.
We applied both the elimination and relaxation approaches, defined by the
optimization problems \eqref{PDE-unconstrained-representer-theorem}
and \eqref{PDE-relaxed-representer-optimization} respectively, for different choices of $M$. In the relaxation approach, \yc{we set $\beta^2=10^{-10}$.}
Here we set $M = 300, 600, 1200, 2400$ and $M_\Omega = 0.9 \times M$.
The $L^2$ and $L^{\infty}$ errors of the converged \bh{Gauss--Newton} solutions are shown in Table \ref{Table:Comparison between the elimination and relaxation approaches}. Results were averaged over 10 realizations of the
random collocation points.  From the table we observe that  the difference in solution errors was very mild and  both methods \bh{were convergent} as $M$ increases.
We note that \bh{in the relaxed setting, convergence} is closely tied to our choice of $\beta$, and choosing
an inadequate value, i.e. too small or too large, can lead to inaccurate solutions. In terms of computational costs, the elimination approaches take 2-3 steps of \bh{Gauss--Newton} iterations on average, while the relaxation approach needs 5-8 steps. Thus while the elimination strategy appears to
be more efficient, we do not observe a significant difference in the order  of complexity of the methods for dealing with the constraints\yc{, especially when the number of collocation points becomes large}.

\begin{table}[ht]
\centering
\yc{
\begin{tabular}{lllll}
\hline
$M$                           & 300        & 600        & 1200       & 2400      \\ \hline
Elimination: $L^2$ error      & 4.84e-02 & 6.20e-05 & 2.74e-06 & 2.83e-07 \\
Elimination: $L^\infty$ error & 3.78e-01 & 9.71e-04 & 4.56e-05& 5.08e-06 \\
Relaxation: $L^2$ error       & 1.15e-01 & 1.15e-04 & 1.87e-06 & 1.68e-07 \\
Relaxation: $L^\infty$ error  & 1.21e+00 & 1.45e-03 & 3.38e-05 & 1.84e-06 \\ \hline
\end{tabular}
}
\caption{Comparison between the elimination and relaxation approaches to
deal with the equality constraints for the nonlinear elliptic PDE \eqref{elliptic-proto-PDE}. Uniformly random collocation points were sampled with different $M$ and $M_{\Omega}=0.9M$. Adaptive nugget terms were employed with the
global nugget parameter \bh{$\eta=10^{-12}$ (see Appendix~\ref{app:Nonlinear-Elliptic-Nugget}).} The lengthscale parameter $\sigma =0.2$. Results were averaged over 10 realizations of the
random collocation points. \yc{The maximum Gauss-Newton iteration was 10.}}
\label{Table:Comparison between the elimination and relaxation approaches}
\end{table}

\subsubsection{Burgers' Equation}\label{sec:Burgers-Revisited}
We consider  numerical solution of the viscous Burgers equation:
\begin{equation}
    \label{Burgers-proto-PDE}
    \begin{aligned}
      \partial_t u +  u \partial_s u  - \nu  \partial_s^2 u  &= 0, \quad \forall
      (s,t) \in (-1, 1)  \times (0,1]\, , \\
      u(s, 0) & = - \sin( \pi x)\, , \\
      u(-1, t) & = u(1, t)  = 0\, .
  \end{aligned}
  \end{equation}
We adopt an approach in which we solve the problem by conditioning a Gaussian process
in space-time\footnote{It would also be possible to look at an incremental in time approach, for example
using backward Euler discretization, in which one iteratively in time solves a nonlinear
elliptic two point boundary value problem by conditioning a spatial Gaussian process; we do not pursue this here and leave it as a future direction.}.
In our experiments we take $\nu =0.02$ and consider $\bx = (s,t)$.
We write this PDE in the form of \eqref{assumptions-on-mP-mB-PDE}
with $Q = 4$ and $Q_b = 1$
with linear operators
$L_1(u)=u, L_2(u)=\partial_t u, L_3(u)=\partial_s u, L_4(u)=\partial_s^2 u$ and the
nonlinear map $P(v_1,v_2,v_3,v_4) = v_2+v_1v_3-\nu v_4^2$. The boundary part is simply  $B(v_1) = v_1$.
 We then eliminate the equality constraints in our optimization framework
 following the approach of Subsection~\ref{sec:eliminating-constraints}
 using the equation $v_2 = \nu v_4^2-v_1v_3$.

We  randomly sampled \yc{$M=2400$ with $M_{\Omega} =2000 $ points} in the computational domain $\Omega = [-1, 1] \times [0,1]$ see  Figure \ref{fig: Burgers demonstration}(a), where we also plot  contours of the true solution $u$. The \bh{Gauss--Newton} algorithm was then
 applied to solve the unconstrained optimization problem. \as{We computed the true solution
 from the Cole--Hopf transformation, together with the numerical quadrature.}
 Since the time and space variability of the solution to Burgers' equation are
 significantly different, we chose an anisotropic kernel
 \footnote{\bh{One can also design/learn a non-stationary kernel using the approaches
 discussed in Subsection~\ref{sec:KernelFlow}. However, the parameterization of
 such kernels and strategies for tuning their hyperparameters are outside the scope
 of this article.}}
 \as{
 $$\cc\Bigl((s,t),(s',t'); \sigma\Bigr) = \exp\Bigl(-\sigma_1^{-2}(s-s')^2-\sigma_2^{-2}(t-t')^2\Bigr)\,$$
 with  $\sigma = (1/20,1/3)$} together \as{with
 an adaptive diagonal regularization (``nugget'') as outlined in Appendix~\ref{app:Burgers-Nugget}.}

 We plot the \bh{Gauss--Newton} iteration history in  Figure \ref{fig: Burgers demonstration}(b) and observe that 10 steps sufficed for convergence.
 We compare the converged solution to the true solution and present the contours of the
 error in Figure \ref{fig: Burgers demonstration}(c). The maximum errors occured close to the (viscous) shock at time $1$ as  expected. In Figure \ref{fig: Burgers demonstration}(d--f), we also compare various time slices of the numerical and true solutions at times  $t=0.2,0.5,0.8$ to further highlight the ability of our method in capturing the
 location and shape of the shock.
    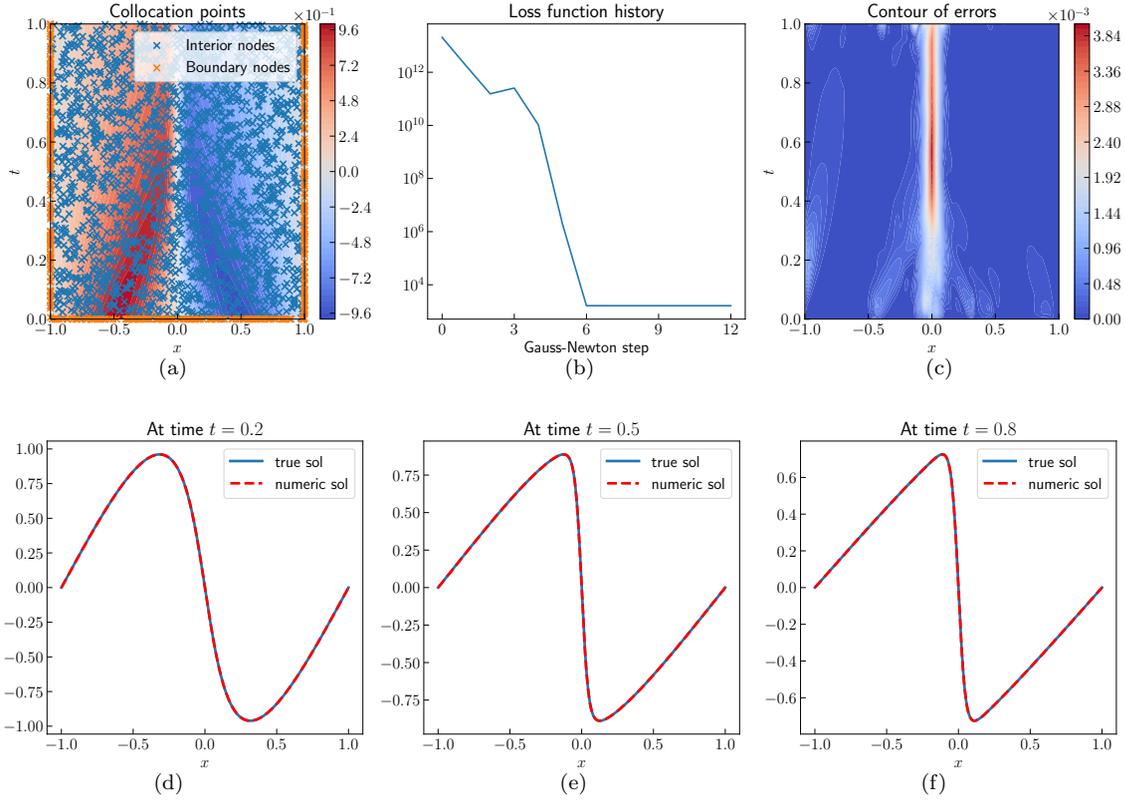
\begin{figure}[ht]
    \centering
    \begin{overpic}[width=15cm]{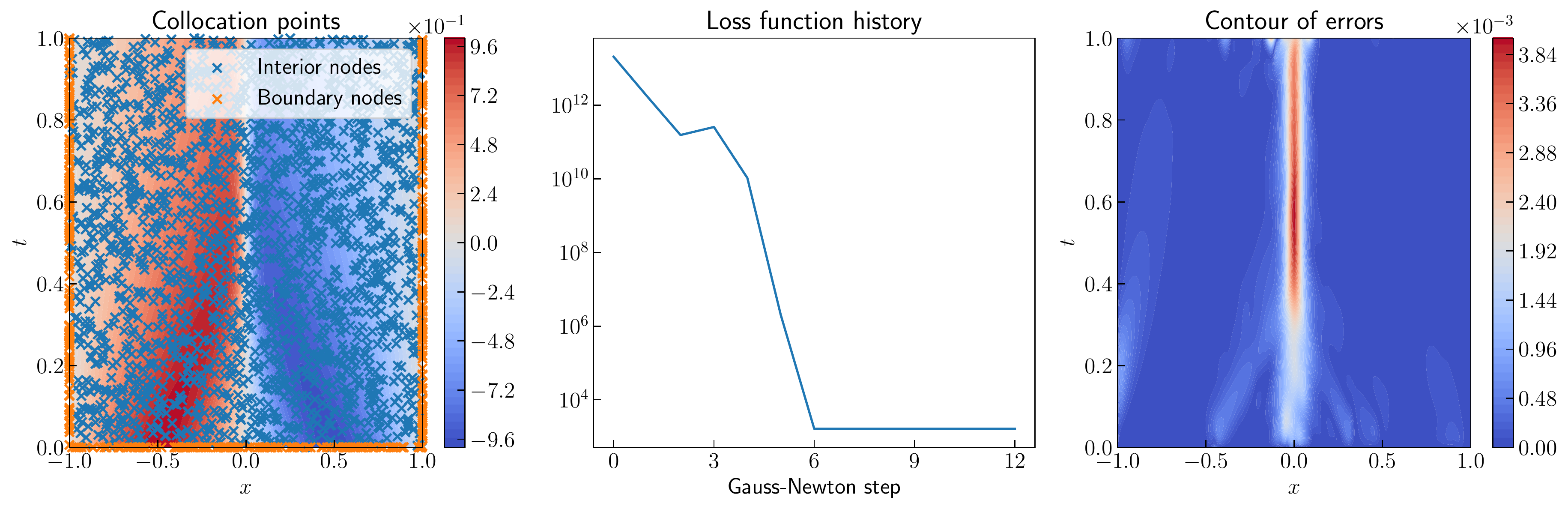}
    \put(14,-1){\scriptsize (a)}
    \put(50,-1){\scriptsize (b)}
    \put(82,-1){\scriptsize (c)}
    \end{overpic}\\
    \vspace{4ex}
    \begin{overpic}[width=15cm]{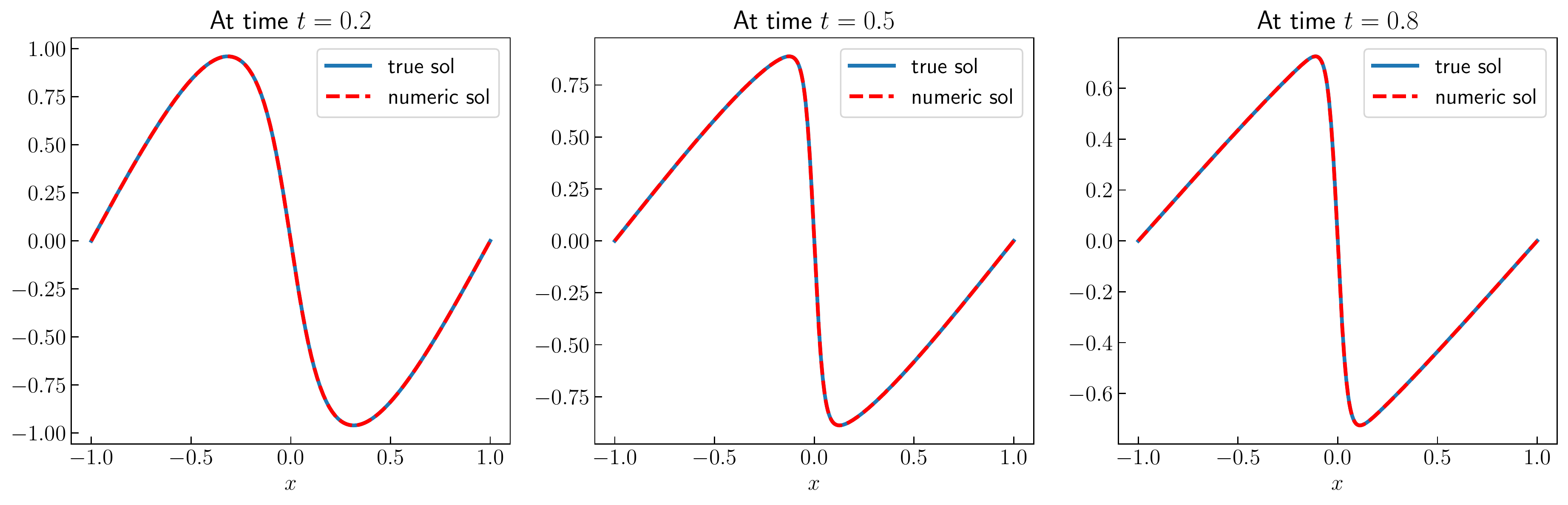}
    \put(14,-1){\scriptsize (d)}
    \put(50,-1){\scriptsize (e)}
    \put(82,-1){\scriptsize (f)}
    \end{overpic}
    \caption{Numerical results for Burgers equation \eqref{Burgers-proto-PDE}: (a)
    an instance of uniformly sampled collocation points in space-time over contours of the true solution; (b)  \bh{Gauss--Newton} iteration history; (c) contours of the pointwise
     error of the numerical solution; (d--f) time slices of the numerical and true
     solutions at $t=0.2,0.5,0.8$. \bh{An adaptive nugget term
    with global parameter $\eta = 10^{-10}$ was employed
    (see Appendix~\ref{app:Burgers-Nugget}).}}
    \label{fig: Burgers demonstration}
    \end{figure}

Next, we studied the  convergence properties  of our method as a function of $M$ as shown in Table \ref{table: Burgers convergence}. Here,  we varied $M$ with a fixed ratio of interior points,
$M_{\Omega}/M=5/6$. For each experiment  we ran $10$ steps of \bh{Gauss--Newton} starting from a Gaussian random initial guess.
Results were averaged over 10 realizations of the
random collocation points.
From the table, we observe that the error decreases very fast as $M$ increases, implying the convergence of our proposed algorithm.

Finally, we note that the accuracy of our method is closely tied to the choice
of the viscosity parameter $\nu$ and choosing a smaller value of $\nu$, which
in turn results in a sharper shock, can significantly reduce our accuracy.
This phenomenon is not surprising since a sharper shock corresponds to
the presence of shorter length and time scales in the solution;
these in turn, require a more careful choice of the kernel, as well as suggesting
the need to carefully choose the collocation points.

\begin{table}[ht]
\centering
\yc{
\begin{tabular}{lllll}
\hline
$M$                 & 600         & 1200       & 2400    & 4800       \\ \hline
$L^2$ error      & 1.75e-02 & 7.90e-03 & 8.65e-04 & 9.76e-05 \\
 $L^\infty$ error  & 6.61e-01 & 6.39e-02  & 5.50e-03 & 7.36e-04 \\ \hline
\end{tabular}
}
\caption{Space-time $L^2$ and $L^{\infty}$ solution errors for the Burgers' equation \eqref{Burgers-proto-PDE} for different choices of $M$ with kernel parameters $\sigma = (20,3)$ and \bh{global nugget parameter $\eta = 10^{-5}$ if $M \leq 1200$ and $\eta = 10^{-10}$ otherwise
(see Appendix~\ref{app:Burgers-Nugget})}. Results were averaged over 10 realizations of the
random collocation points. \yc{The maximum Gauss-Newton iteration was 30.}}
\label{table: Burgers convergence}
\end{table}

\subsubsection{Eikonal PDE}\label{sec:eikonal-PDE}
We now  consider the regularized Eikonal equation in $\Omega = [0,1]^2$:
\begin{equation}
  \label{reg-Eikonal}
  \left\{
    \begin{aligned}
      |\nabla u(\bx)|^2 &= f(\bx)^2+\epsilon \Delta u(\bx), && \forall
      \bx \in \Omega\,, \\
      u(\bx) & = 0, && \forall
      \bx \in \partial\Omega\, ,
    \end{aligned}
    \right.
\end{equation}
with $f = 1$ and $\epsilon=0.1$.
We write this PDE in the form of \eqref{assumptions-on-mP-mB-PDE} with
$Q = 4$ and $Q_b = 1$ and
linear operators $L_1(u)=u, L_2(u)=\partial_{x_1} u, L_3(u)=\partial_{x_2} u, L_4(u)=\Delta u$ and nonlinear map $P(v_1,v_2,v_3,v_4) = v_2^2+v_3^2-\epsilon v_4$ in the interior of
$\Omega$ and define the boundary operator identically to Subsection~\ref{sec:Burgers-Revisited}.
We further eliminate the  nonlinear constraints, as outlined in Subsection~\ref{sec:eliminating-constraints},
by solving $v_4$ in terms of $v_2,v_3$.
To obtain a ``true'' solution, for the purpose of estimating errors,
we employ the transformation $u = -\epsilon\log v$, which leads to the linear PDE $fv-\epsilon^2\Delta v = 0$; we
solve this by a highly-resolved FD method\yc{; we used $2000$ uniform grid points in each dimension of the domain leading to the finest mesh that our hardware could handle}.

As before, we began with \yc{$M=2400$ collocation points with $M_{\Omega} = 2160$} interior
points. An instance of these collocation points along with contours of the true
solution are shown in Figure~\ref{fig:Eikonal}(a). We employed a
nugget term  as \bh{outlined in Appendix~\ref{app:Eikonal-Nugget}} and used the Gaussian kernel,
as in Subsection~\ref{sec:elliptic-PDE} with $\sigma = M^{-1/4}$. Finally we used the
 \bh{Gauss--Newton} algorithm to find the minimizer. We show the
 convergence history of \bh{Gauss--Newton} in Figure~\ref{fig:Eikonal}(b), observing
 that six iterations were sufficient for convergence. In Figure~\ref{fig:Eikonal}(c)
 we show the error contours of the obtained numerical approximation, which appeared
 to be qualitatively different to Figure~\ref{fig:nonlinear elliptic, demon}(c)
 in that the errors \bh{were larger   in the middle of the domain as well as  close to the
 boundary. }

\begin{figure}[ht]
    \centering
    \begin{overpic}[width=15cm]{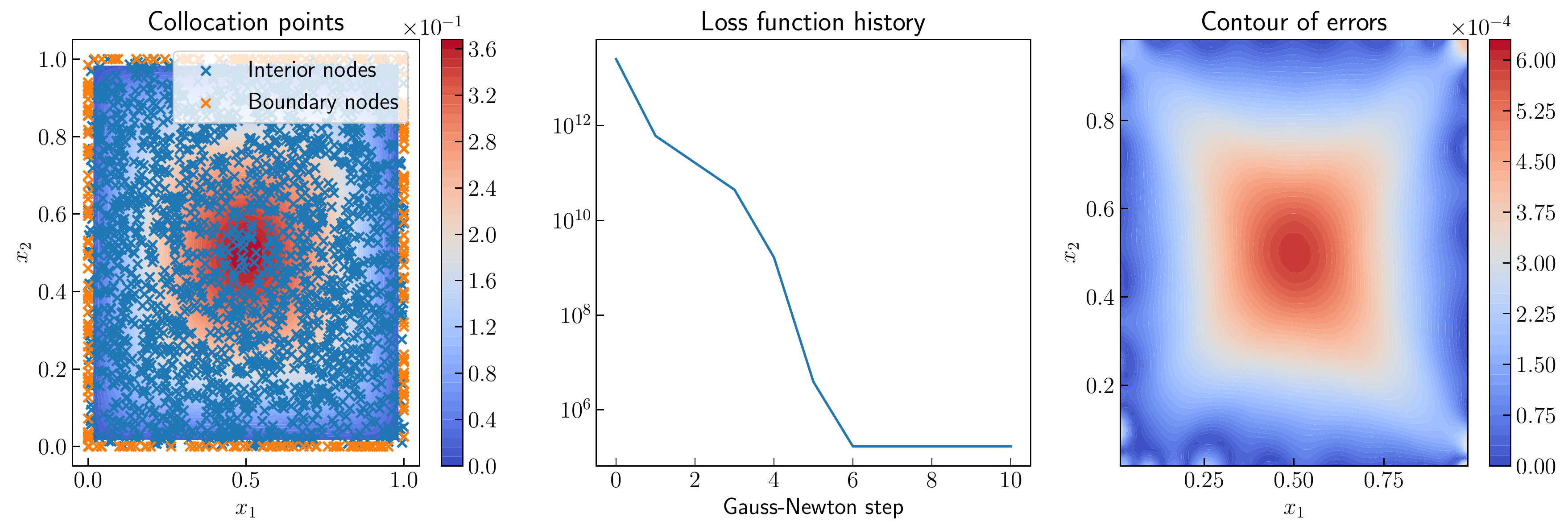}
    \put(14,-1){\scriptsize (a)}
    \put(50,-1){\scriptsize (b)}
    \put(82,-1){\scriptsize (c)}
    \end{overpic}
    \caption{Numerical results for the regularized Eikonal equation \eqref{reg-Eikonal}: (a) an instance of uniformly sampled collocation points over contours of the true solution; (b) \bh{Gauss--Newton} iteration history; (c) contour of the solution error.
    \bh{An adaptive nugget term with $\eta =10^{-10}$ was used (see Appendix~\ref{app:Eikonal-Nugget}).}
    }
    \label{fig:Eikonal}
\end{figure}

Next we performed a convergence study by varying  $M$ and computing $L^2$ and $L^{\infty}$ errors as reported  in Table \ref{table: Eikonal equation, error}
by choosing the lengthscale parameter of the kernel $\sigma = M^{-1/4}$. \bh{We used the
same  nugget terms as in the Burgers' equation (see Appendix~\ref{app:Eikonal-Nugget})}.
Results were averaged over 10 realizations of the
random collocation points. \yc{Once again we observe clear improvements in accuracy as the number of collocation points increases.}

\begin{table}[ht]
\centering
\yc{
\begin{tabular}{lllll}
\hline
$M$                 & 300       &600     &1200      & 2400       \\ \hline
$L^2$ error      & 1.01e-01 & 1.64e-02 & 2.27e-04 & 7.78e-05 \\
 $L^\infty$ error  & 3.59e-01 & 7.76e-02  & 3.22e-03 & 1.61e-03 \\ \hline
\end{tabular}
}
\caption{Numerical results for the regularized Eikonal equation \eqref{reg-Eikonal}. Uniformly random collocation points were sampled with different $M$ and with fixed ratio $M_\Omega=0.9M$. An adaptive nugget term was used with \bh{global nugget parameter $\eta = 10^{-5}$ if $M \leq 1200$ and $\eta = 10^{-10}$ otherwise (see Appendix~\ref{app:Eikonal-Nugget})}, together with a Gaussian kernel with
lengthscale parameter $\sigma = M^{-1/4}$.
Results were averaged over 10 realizations of the
random collocation points. \yc{The maximum Gauss-Newton iteration was  20.}}
\label{table: Eikonal equation, error}
\end{table}

\section{Solving Inverse Problems}\label{sec:Inverse-Problems}
We now turn our attention to solution  of IPs and show that
the methodology of Subsections~\ref{sec:nonlinear-PDE-GP}--\ref{sec:implementation}
can readily be extended to solve such problems with small modifications. We descibe the
abstract setting of our IPs in Subsection~\ref{sec:IP-setup} leading to Corollary~\ref{cor:IP-representer-theorem}
which is analogous to  Proposition~\ref{prop:nonlinear-representer-optimization} and Corollary~\ref{cor:PDE-representer-theorem}
in the setting of IPs. Subsection~\ref{sec:dealing-with-constraints-IP} outlines our approach for
dealing with PDE constraints in IPs and highlights the differences in this setting in comparison to the PDE
setting described in Subsection~\ref{sec:deadling-with-constraints}. Subsection~\ref{sec:implementation-for-IPs}
further highlights the implementation differences between the PDE and IP settings while Subsection~\ref{sec:Darcy-flow}
presents a numerical experiment concerning an IP in  subsurface flow governed by the Darcy flow PDE.

\subsection{Problem Setup}\label{sec:IP-setup}
 Consider our usual setting of a nonlinear parameteric PDE in strong form
\begin{equation}
  \label{prototypical-parametric-PDE}
  \left\{
  \begin{aligned}
    \mP(u^\star; a^\star) (\bx) = f(\bx)
    ,& \qquad \forall \bx \in \Omega\, , \\
  \mB(u^\star; a^\star) (\bx) = g(\bx),& \qquad \forall \bx \in \partial \Omega\, .
  \end{aligned}
  \right.
\end{equation}
As before we assume the solution $u^\star$ belongs to a quadratic
Banach space $\mU$  while $a^\star$ is a parameter belonging to another
quadratic Banach space $\mA$.
Our goal in this subsection is to identify the parameter $a^\star$ from limited
observations of the solution $u^\star$. To this end, fix $\psi_1, \dots, \psi_I \in \mU^\ast$
and define
\begin{equation}
  \label{Gamma-definition}
   \bpsi := (\psi_1, \dots, \psi_I) \in (\mU^\ast)^{\otimes I}\, ,
\end{equation}
then our goal is to recover $a^\star$ given the noisy observations
\begin{equation}
  \label{noisy-observation-model}
  \bo = [\bpsi, u]  + \bm{\eps}, \qquad \bm{\eps}\sim \mN(0, \gamma^2I)\, .
\end{equation}

We propose to solve this inverse problem by modelling both $u^\star$ and $a^\star$ with canonical GPs
on the spaces $\mU, \mA$
with invertible covariance operators $\C: \mU^\ast \to \mU$ and $\S: \mA^\ast \to \mA$
respectively. We then condition these GPs to satisfy the PDE on collocation points
$\bx_1, \dots, \bx_M \in \overline{\Omega}$ as before and propose to
approximate  $u^\star, a^\star$ simultaneously via the
optimization problem:
\begin{equation}\label{inverse-problem-optimization-form-initial}
    \left\{
      \begin{aligned}
        & \minimize_{(u,a) \in \mU \times \mA}  &&   \| u \|_\mU^2  + \| a \|_\mA^2
        + \frac{1}{\gamma^2} \big| [\bpsi, u] - \bo \big|^2 &&  \\
        & \st  &&   \mP(u;a)(\bx_m) = f(\bx_m),  \qquad &&\text{for } m = 1, \dots, \Md\, ,\\
        &       &&  \mB(u;a)(\bx_m) = g(\bx_m),  \qquad &&\text{for } m = \Md+1, \dots, M\, ,\\
      \end{aligned}
      \right.
    \end{equation}
    where we used  subscripts to distinguish the quadratic norms on the spaces $\mU$ and $\mA$.

    \begin{remark}
    \bh{
    In light of Remark~\ref{rem:Bayesian-interpretation} we
    observe that \eqref{inverse-problem-optimization-form-initial} corresponds to
    imposing a prior measure on $u, a$ which assumes they are a priori independent.
    \as{It is straightforward to} introduce correlations between the solution $u$ and the parameter $a$
    by defining the prior measure directly on the product space $\mU \times \mA$.
    This perspective will then lead to an analogous optimization problem to
    \eqref{inverse-problem-optimization-form-initial} with the same constraints but with
    the functional
    \begin{equation*}
        \| (u, a) \|_{\mU \times \mA}^2 + \frac{1}{\gamma^2} | [\bpsi, u] - \bo |^2,
    \end{equation*}
    where we used $\| \cdot \|_{\mU \times \mA}$ to denote the RKHS norm of the  GP
    associated with  $\mU \times \mA$.
    }
    \end{remark}

    \begin{remark}
    \bh{
    We also note that the Bayesian interpretation of  \eqref{inverse-problem-optimization-form-initial} can be viewed as an extension of
    gradient matching \cite{calderhead2009accelerating, liang2008parameter} from
    ODEs to PDEs. Indeed, gradient matching simultaneously approximates the unknown parameters and the
    solution of an ODE system using a joint GP prior and imposes the ODE as a
    constraint at finitely many time steps.
    }
    \end{remark}

 We make analogous assumptions on the form of the operators $\mP, \mB$ as in
 Assumption~\ref{PDE-assumption} but this time also involving the parameters $a$:

\begin{assumption}\label{Inverse-problem-assumption}
There exist bounded and linear operators $L_1, \dots, L_{Q} \in \mL( \mU; C( \Omega))$ in which $L_{1}, \dots, L_{Q_b} \in \mL( \mU; C( \partial \Omega))$ for some $1\leq Q_b \leq Q$,
and
 $\tL_{1}, \dots, \tL_{J} \in \mL(\mA; C(\overline{\Omega}))$
together with  maps $P: \R^{Q+J} \to \R$ and $B: \R^{Q_b+J} \to \R$, \bh{which may be nonlinear,}
so that $\mP(u;a)( \bx)$
and $\mB(u;a) (\bx)$ can be written as
\begin{equation}
  \label{assumptions-on-mP-mB-Inverse-Problem}
  \begin{aligned}
  \mP(u;a)(\bx) & = P\big( L_1(u)(\bx), \dots, L_{Q}(u)(\bx); \tL_1(a)(\bx), \dots \tL_J(a)(\bx) \big),
  && \forall \bx \in \Omega\, , \\
  \mB(u;a)(\bx) & = B\big( L_{1}(u)(\bx), \dots, L_{Q_b}(u)(\bx); \tL_1(a)(\bx), \dots \tL_J(a)(\bx) \big),
  && \forall \bx \in \partial\Omega\, .
\end{aligned}
\end{equation}
\end{assumption}

Similarly to the $L_q$, the $\tL_j$ operators are  also linear differential operators in case of prototypical
PDEs  while the maps $P, B$ remain as simple algebraic nonlinearities. Let us briefly consider
an IP in subsurface flow and verify the above assumption.

\begin{runningexample}{DF}[Darcy flow IP]
\label{example-darcy-flow}
  Let $\Omega = (0,1)^2$ and consider the Darcy flow PDE with Dirichlet boundary conditions
  \begin{equation*}
    \left\{
    \begin{aligned}
      - \Div \left(   \exp(a) \nabla u \right) (\bx)  & = f (\bx),  && \bx \in \Omega\, , \\
          u(\bx) & = 0,   && \bx \in \partial \Omega\, .
  \end{aligned}
  \right.
\end{equation*}
We wish to approximate $a \in C^1(\overline{\Omega})$ given noisy pointwise observations of $u$ at a set of points
$\tilde{\bx}_{1}, \dots, \tilde{\bx}_I$. Thus, we take $\psi_i = \updelta_{\tilde{\bx}_i}$. By expanding the
PDE we obtain
\begin{equation*}
  - \Div \left(   \exp(a) \nabla u \right) = - \exp(a) \left(  \nabla a \cdot \nabla u + \Delta u \right)\, ,
\end{equation*}
and so we simply choose $Q = 3$, $Q_b = 1$ and $J=2$ with the linear operators
\begin{equation*}
  L_1(u) = u, \quad  L_2(u) = \nabla u, \quad
  L_3(u) = \Delta u, \quad
  \tL_1(a) = a, \quad \tL_2(a) = \nabla a\, .
\end{equation*}
We can then satisfy Assumption~\ref{assumptions-on-mP-mB-Inverse-Problem}
by taking
\begin{equation}
\label{eqn: Darcy flow variables}
  P(v_1, \bv_2, v_3; v_4, \bv_5) = - \exp(v_4) \left( \bv_5 \cdot \bv_2 + v_3 \right), \qquad
  B(v_1; v_4, \bv_5) = v_1\, ,
\end{equation}
where we have slightly abused notation by letting $L_2, \tL_2$  be vector valued and
defining $P,B$ to take  vectors as some of their inputs.
\end{runningexample}

As in Subsection~\ref{sec:nonlinear-PDE-GP} we now define
the  functionals
$\phi_{m}^{(q)} \in \mU^\ast$ for $m=1, \dots, M$ and $q  = 1, \dots, Q$ according to \eqref{phi-definition}
and  \eqref{phi-PDE-def}
with $N = MQ_b+M_{\Omega}(Q-Q_b)$.
Similarly we define the  functionals $\tphi_{m}^{(j)} \in \mA^\ast$  as
\begin{equation}\label{varphi-definition-inverse-problem}
  \tphi_{m}^{(j)} := \delta_{\bx_m} \circ \tL_j, \qquad  \text{for} \qquad m = 1, \dots, M, \text{ and }
  j = 1, \dots, J\, ,
\end{equation}
together with the vectors
\begin{equation}
  \label{tilde-phi-definition}
  \tbphi^{(j)} = (\tphi_1^{(j)}, \dots \tphi_M^{(j)} ) \in ( \mA^\ast)^{\otimes M} \quad \text{and} \quad
  \tbphi= (\tbphi^{(1)}, \dots, \tbphi^{(J)} ) \in  (\mA^\ast)^{\otimes \tN}\, ,
\end{equation}
where $\tN := M J$.
Similarly to  \eqref{F-map-PDE-definition} define  the map
\begin{equation*}
  \begin{aligned}
    \big( &F(  [\bphi, u]_\mU ; [\tbphi, a]_\mA ) \big)_m := \\
    &\left\{
    \begin{aligned}
      &P( [\phi_m^{(1)}, u]_\mU,  \dots, [\phi_m^{(Q)}, u]_\mU;
      [ \tphi_{m}^{(1)}, a]_\mA, \dots, [\tphi_{m}^{(J)}, a]_\mA   )  &&\text{if }
      m \in \{1, \dots, \Md\}\, , \\
      &B([\phi_m^{(1)}, u]_\mU,  \dots, [\phi_m^{(Q_b)}, u]_\mU;
      [\tphi_{m}^{(1)},a ]_\mA, \dots, [\tphi_{m}^{(J)}, a]_\mA)
      &&\text{if } m \in \{ \Md + 1, \dots, M\}\, ,
    \end{aligned}
    \right.
  \end{aligned}
  \end{equation*}
  where  we used subscripts to distinguish the duality pairings between $\mU, \mU^\ast$
  and the pairing between $\mA, \mA^\ast$. With this new notation we can finally
 rewrite \eqref{inverse-problem-optimization-form-initial} in the  familiar form
\begin{equation}\label{inverse-problem-optimization-form-standard}
    \left\{
      \begin{aligned}
        & \minimize_{(u,a) \in \mU \times \mA}
        \quad  && \| u\|_\mU^2  +  \| a\|_\mA^2
        + \frac{1}{\gamma^2} | \bpsi(u) - \bo |^2   \\
        & \st \quad   && F \big([\bphi, u]_\mU; [\tbphi, a]_\mA \big) = \by,
      \end{aligned}
      \right.
\end{equation}
with the PDE data vector $\by \in \R^M$ defined in \eqref{y-definition}.

We can now apply Proposition~\ref{prop:nonlinear-representer-optimization} with the canonical
GP defined
on the product space $\mU \times \mA$ and with a block diagonal covariance operator $\mK \otimes \S$ to obtain a representer theorem for minimizer of
\eqref{inverse-problem-optimization-form-standard}. We state this result as a corollary below after
introducing some further notation.
Define the vector
$\bvphi = ( \vphi_1, \dots, \vphi_{N + I}) \in ( \mU^\ast)^{\otimes (I + N)},$
with entries \footnote{Note that we are concatenating the $I$ measurement functionals defining the data for the IP
with the $N$ functionals used to define the PDE at the collocation points.}
\begin{equation*}
  \vphi_n := \left\{
    \begin{aligned}
      &\psi_n,  &&\text{if }n = 1, \dots, I, \\
      &\phi_{n - I},  &&\text{if }n = I+1, \dots, I+N,
    \end{aligned}
    \right.
\end{equation*}
as well as  the matrices $\Theta \in \R^{(I + N) \times (I + N)}$ and $\tTheta \in \R^{\tN \times \tN}$ with entries
\begin{equation*}
  \Theta_{i,n} = [ \vphi_i, \C \vphi_n ]_{\mU} \qquad \text{and} \qquad
  \tTheta_{i,n} = [\tphi_i, \S \tphi_n ]_{\mA}.
\end{equation*}
As in \eqref{def:gamblets} we define the
 gamblets
   \begin{equation*}
     \chi_i = \sum_{n=1}^{N+I} \Theta_{i, n}^{-1} \C \vphi_n, \qquad \text{and} \qquad
     \tchi_i = \sum_{n=1}^{\tN} \tTheta_{i, n}^{-1} \S \tphi_n.
   \end{equation*}
Then Proposition~\ref{prop:nonlinear-representer-optimization} yields the following corollary.

 \begin{corollary}\label{cor:IP-representer-theorem}
   Suppose Assumption~\ref{Inverse-problem-assumption} holds and that the
   covariance operators $\C$ and $\S$ as well as the matrices $\Theta$ and $\tTheta$
   are invertible.
   Then   $(u^\dagger, a^\dagger) \in \mU \times \mA$ is a minimizer of \eqref{PDE-optimization-form-initial} if and only if
   \begin{equation*}
     u^\dagger = \sum_{n=1}^{I + N} z^\dagger_n \chi_n, \qquad \text{and} \qquad a^\dagger = \sum_{n=1}^{\tN}
     \tz^\dagger_n \tchi_n,
   \end{equation*}
   where the vectors $\bz^\dagger, \tbz^\dagger$ are minimizers of
   \begin{equation}\label{IP-nonlinear-representer-optimization}
         \left\{
        \begin{aligned}
          & \minimize_{(\bz, \tbz) \in (\R^{I + N} \times \R^{\tN}) }
          &&    \bz^T \Theta^{-1} \bz
          +  \tbz^T \tTheta^{-1} \tbz
          + \frac{1}{\gamma^2}  | \Pi^I \bz - \bo |^2 \\
        & \st  && F( \Pi_N \bz; \tbz) = \by,
      \end{aligned}
      \right.
    \end{equation}
    where $\Pi^I: \R^{I + N} \to \R^I$ is the projection that extracts the first $I$ entries
    of a vector while
    $\Pi_N: \R^{I + N} \to \R^{N}$ is the projection  that extracts the
    last $N$ entries.
  \end{corollary}

    \subsection{Dealing with the Constraints}\label{sec:dealing-with-constraints-IP}
    The equality constraints in \eqref{IP-nonlinear-representer-optimization} can be dealt with using
    the same strategies as in Subsection~\ref{sec:deadling-with-constraints}. Indeed, as in
    Subsection~\ref{sec:relaxed-constraints}, we can
    readily relax these constraints to obtain the optimization problem
    \begin{equation}\label{IP-relaxed-representer}
        \begin{aligned}
          & \minimize_{(\bz, \tbz) \in (\R^{I + N} \times \R^{\tN}) }
          &&    \bz^T \Theta^{-1} \bz
          +  \tbz^T \tTheta^{-1} \tbz
          + \frac{1}{\gamma^2}  | \Pi^I \bz - \bo |^2 + \frac{1}{\beta^2} \left|  F( \Pi_N \bz; \tbz) - \by \right|^2,
      \end{aligned}
    \end{equation}
    for a small parameter $\beta^2> 0$. Elimination of the constraints
    as in Subsection~\ref{sec:eliminating-constraints} is slightly more delicate, but is sometimes
    possible. Suppose there exists
    a solution map $\hF: \R^{N + \tN - M} \times \R^{M} \to \R^{N + \tN}$ so that
    \begin{equation*}
        F( \Pi_N \bz; \tbz) = \by \quad \text{if and only if} \quad (\Pi_N \bz, \tbz) = \hF(\bw, \by) \qquad \text{for a unique }
        \bw \in \R^{N + \tN - M}.
    \end{equation*}
    Then solving \eqref{IP-nonlinear-representer-optimization} is equivalent to solving the unconstrained problem
      \begin{equation}\label{IP-unconstrained-representer}
        \begin{aligned}
          & \minimize_{ (\bv, \bw) \in \R^I \times \R^{N + \tN - M} }
          && ( \bv, \hF(\bw, \by) )
          \begin{bmatrix}
            \Theta^{-1} & 0 \\
            0 & \tTheta^{-1}
          \end{bmatrix}
          \begin{pmatrix}
            \bv \\
            \hF(\bw, \by)
          \end{pmatrix}
          + \frac{1}{\gamma^2}  | \bv - \bo |^2,
      \end{aligned}
    \end{equation}
    and setting $\Pi^I \bz^\dagger = \bv^\dagger$ and $(\Pi_N \bz^\dagger, \tbz) = \hF(\bw^\dagger, \by)$.

    \subsection{Implementation}\label{sec:implementation-for-IPs}
    Both of the problems \eqref{IP-relaxed-representer} and \eqref{IP-unconstrained-representer} can be solved
    using the same techniques outlined in Subsection~\ref{sec:implementation} except that we now have
    a higher dimensional solution space. Below we briefly describe the main differences between the  implementation of
    the PDE and IP solvers.

    \subsubsection{Constructing $\Theta$ and $\tTheta$}
    We propose to construct the matrices $\Theta, \tTheta$ using appropriate kernels
    $K$, for the solution $u$ of the PDE, and $\tK$, for the parameter $a$ identically to Subsection~\ref{sec:constructing-theta}.
    Our minimum requirements on $K, \tK$ is sufficient regularity for the pointwise constraints in \eqref{inverse-problem-optimization-form-initial} to be well-defined. Since we have limited and noisy data in the inverse
    problem setting, it is not possible for us to recover the exact solution $(u^\star, a^\star)$ in general
    and so the kernels $K, \tK$ should be chosen to reflect our prior assumptions on the unknown parameter
    and the solution of the PDE at that parameter value.

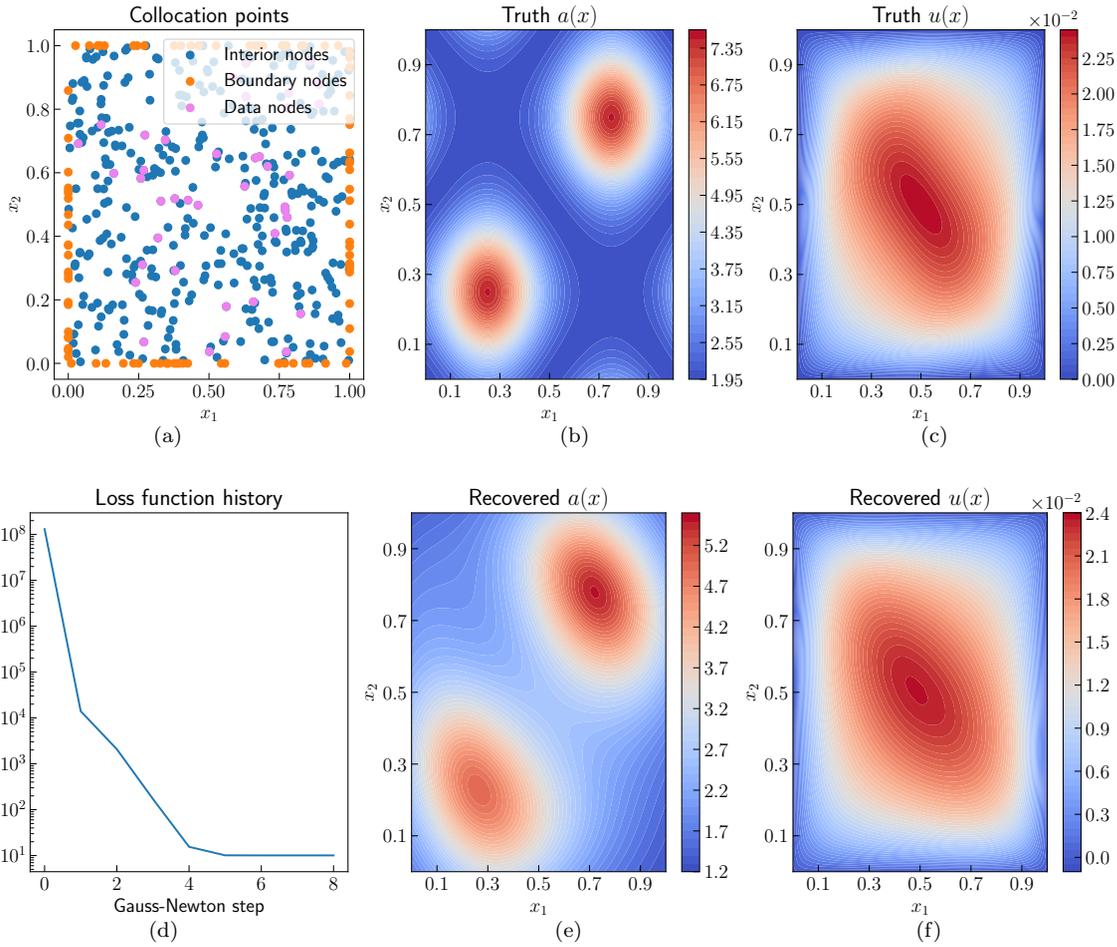
\begin{figure}[ht]
    \centering
    \begin{overpic}[width=15cm]{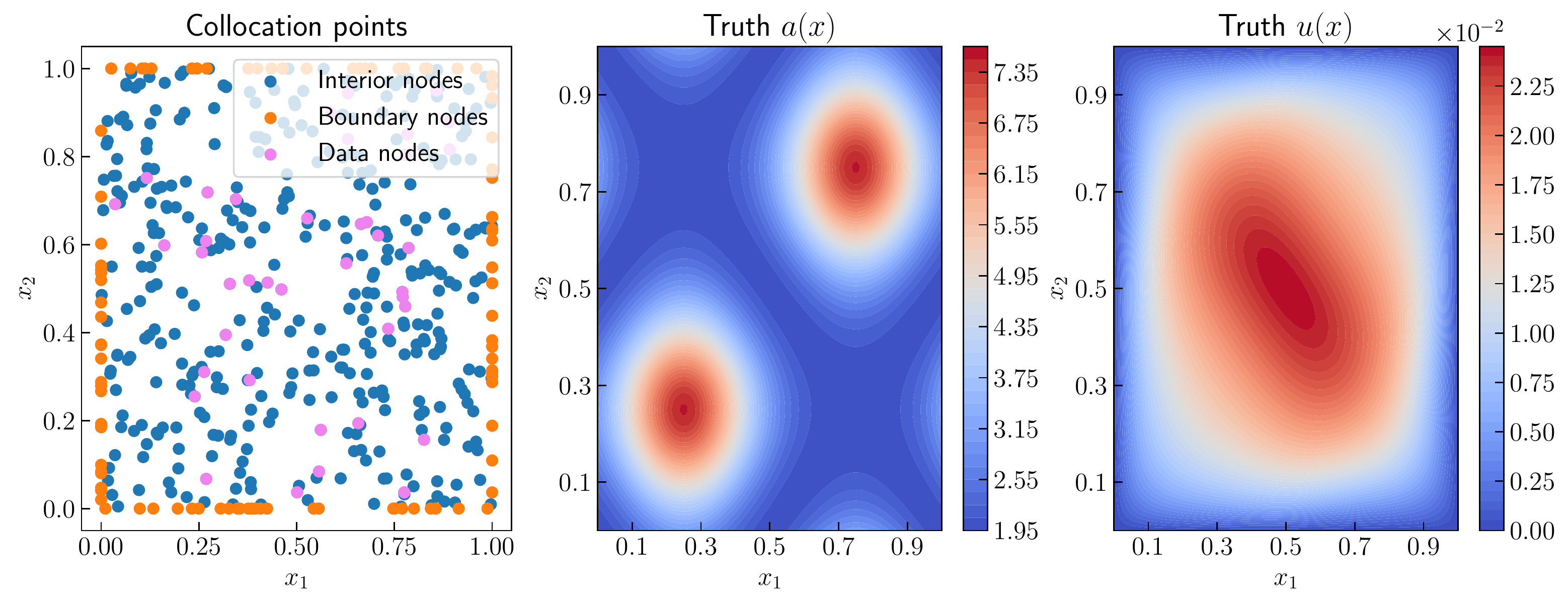}
    \put(14,-1){\scriptsize (a)}
    \put(50,-1){\scriptsize (b)}
    \put(82,-1){\scriptsize (c)}
    \end{overpic}\\
    \vspace{4ex}
    \begin{overpic}[width=15cm]{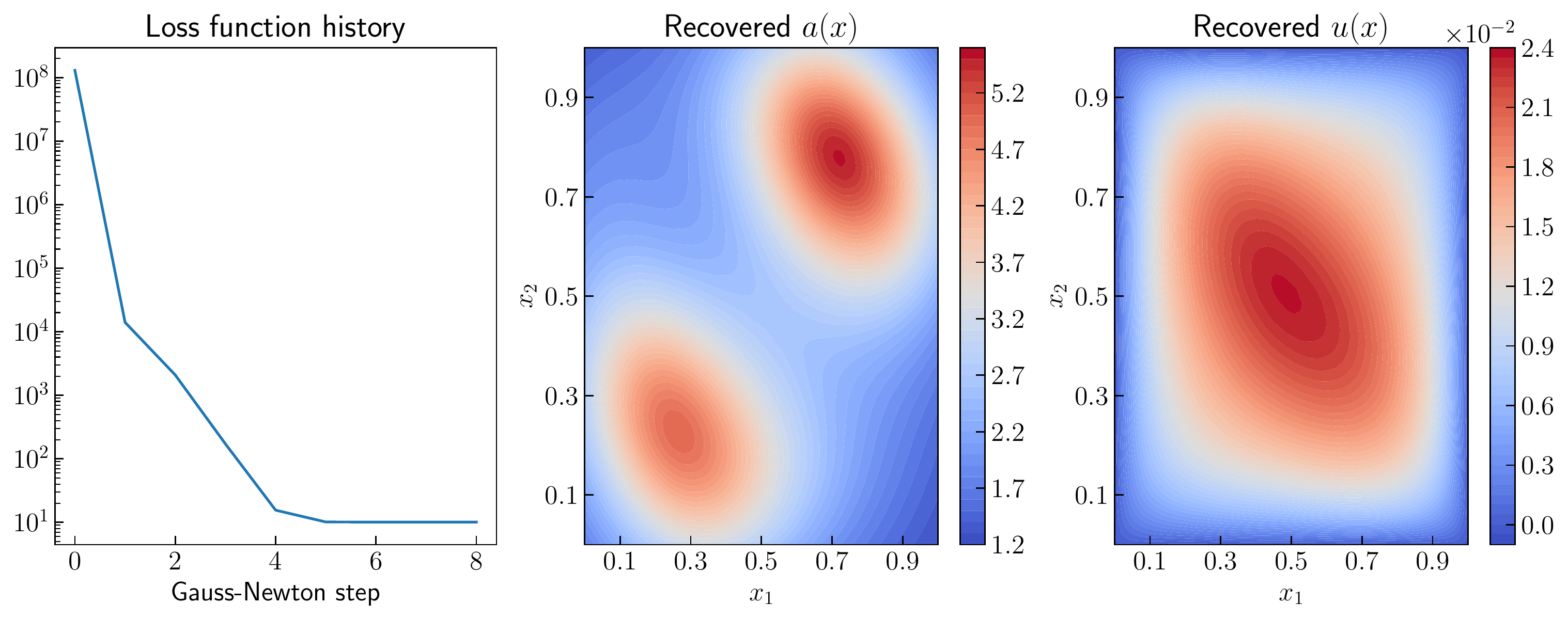}
    \put(14,-1){\scriptsize (d)}
    \put(50,-1){\scriptsize (e)}
    \put(82,-1){\scriptsize (f)}
    \end{overpic}
    \caption{Numerical results for the inverse Darcy flow: (a) an instance of uniformly sampled collocation points and data points; (d) \bh{Gauss--Newton} iteration history; (b) true $a$; (e) recovered $a$; (c) true $u$; (f) recovered $u$. \as{Adaptive diagonal regularization (``nugget'') terms
    were added to the kernel matrix, with parameters $\eta = \tilde \eta = 10^{-5}$ as outlined in Appendix~\ref{app:Darcy-Nugget}.}}
    \label{fig:Darcy flow}
\end{figure}

\subsection{Numerical Experiments for Darcy Flow}\label{sec:Darcy-flow}
In this subsection, we apply our method to an IP involving the  Darcy flow PDE.
We consider the IP outlined in Example \ref{example-darcy-flow}
with the true coefficient $a^\star(\bx)$ satisfying
\[\exp\Bigl(a^\star(\bx)\Bigr) = \exp\Bigl(\sin(2\pi x_1)+\sin(2\pi x_2)\Bigr)+\exp\Bigl(-\sin(2\pi x_1)-\sin(2\pi x_2)\Bigr)\, , \]
and the right hand side source term  is $f=1$. We randomly sampled $M=500$ collocation points with $M_\Omega = 400$ in the interior. From these $400$ interior points, we randomly chose $I = 40$ points
and observed the values of $u(\bx)$ at those points as the data for the IP.
The values of $u(\bx)$ for this purpose were generated by first solving the equation
with the true coefficient on a uniform grid and then using linear interpolation to get the solution at the observation points.
We further added independent Gaussian noise $\mathcal{N}(0,\gamma^2 I)$ with
noise standard deviation $\gamma = 10^{-3}$ to these observations.  In dealing with the nonlinear constraint shown in Example~\ref{example-darcy-flow}, we eliminated the variable $v_3$ using the relation in \eqref{eqn: Darcy flow variables}.

We chose Gaussian kernels for both $u$ and $a$ with the same lengthscale parameter $\sigma =0.2$ and \as{adaptive diagonal  (``nugget'') terms
    were added to the kernel matrices, with parameters   $\eta = \tilde \eta = 10^{-5},$
to regularize $\Theta$ and $\tTheta$ as outlined in Appendix~\ref{app:Darcy-Nugget}.} In Figure \ref{fig:Darcy flow} we show the experimental results for recovering both $a$ and $u$.
From the figure, we observe that the \bh{Gauss--Newton} iterations converged after 6 steps. Moreover, the recovered $a$ and $u$ are reasonably accurate, i.e. they capture the shape of the truth, given the limited amount of observation information available.

\section{Concluding Remarks}\label{sec:discussion}
We \as{have} introduced a kernel/GP framework for solving nonlinear PDEs and IPs centered around the idea of  approximating the solution of a given PDE with a MAP estimator of a GP conditioned on satisfying the PDE at a set of collocation points. Theoretically, we exhibited
a nonlinear representer theorem which finite-dimensionalizes the MAP estimation problem and
proved the convergence of the resulting solution towards the truth as the number of collocation points goes to infinity, under some regularity assumptions. Computationally, we demonstrated that the solution can be found by solving a finite-dimensional optimization problem with quadratic loss and nonlinear constraints. We presented two methods for dealing with the nonlinear constraints, namely the elimination and relaxation approaches. An efficient
variant of the \bh{Gauss--Newton}  algorithm was also proposed for solving the resulting unconstrained optimization problem, where an adaptive nugget term
was employed for regularization together with offline Cholesky factorizations of the underlying kernel matrices. We demonstrated that  the
proposed algorithm performs well in a wide range of prototypical nonlinear problems such as a nonlinear elliptic PDE, Burgers' equation, a regularized Eikonal equation, and the identification of the
permeability field  in Darcy flow.

While our theoretical results  established the existence of a solution for the finite dimensional optimization problem and the convergence of $u^\dagger$ to the truth, the uniqueness of the solution is not guaranteed and the convergence theorem does not imply convergence rates. In what follows,
we provide several additional discussions that hold the potential to address these two issues and point towards interesting future directions: in Subsection \ref{sec-uniqueness-optimality} we discuss how to get the uniqueness result if an appropriate condition is assumed.
In Subsection \ref{sec:error-estimates},  we lay a path to obtain rigorous error estimates followed by a discussion about
learning of hierarchical kernel parameters in Subsection \ref{sec:KernelFlow}.
Finally, in Subsection \ref{subsecnongaussian}, we connect the  framework of this article to MAP estimation for Bayesian  IPs with non-Gaussian priors.

\subsection{Uniqueness and Optimality of  Minimizers} \label{sec-uniqueness-optimality}
While our theoretical results in Subsection~\ref{sec:rep-theorems-and-constraints} characterize the minimizers of
nonlinearly constrained optimization problems in RKHSs they do not imply the  uniqueness
of such minimizers. The lack of convexity  due to the nonlinearity of $F$ makes it impossible to obtain
such uniqueness results, however,  a positive result can be obtained under appropriate assumptions
on the function $F$.
Recall  \eqref{PDE-nonlinear-representer-optimization},
and assume that the following condition holds
\begin{equation}\label{eqconditionF}
(\bz-\bz^\dagger)^T\Theta^{-1} \bz^\dagger\geq 0 \text{ for }F(\bz)=\by.
\end{equation}
Generally, this
condition is equivalent to the requirement that the  nonlinear manifold $ \mV:=\big\{v \in \mU\mid F\big([\bphi,v]\big)=F\big([\bphi,u^\star]\big)\big\}$
 lies on only one side of the tangent hyperplane supported  by $\bz^\dagger$ with $\Theta^{-1}$-normal direction $\bz^\dagger$.
Note that \eqref{eqconditionF} is always satisfied with an equality if the PDE is linear, i.e., $F$ is affine.
Then the following proposition states that the minimizer
 $u^\dagger$ of \eqref{PDE-optimization-form-initial}
is not only unique, it is also the minimax optimal approximation of
$u^\star$
in $\mV$ in the sense of optimal recovery \cite{micchelli1977survey}.

\begin{proposition}
Under condition \eqref{eqconditionF}  it holds true that:
\begin{enumerate}
\item $u^\dagger$ is unique.
\item  $\|v-u^\dagger\|^2\leq \|v\|^2-\|u^\dagger\|^2$ for any $v\in \mV$.
\item
$u^\dagger$ is one of the minimzers of the minimax loss $\inf_{u\in \mV } \sup_{v\in \mV, v\not=0} \frac{\|v-u\|}{\|v\|}$.
\end{enumerate}
\end{proposition}
\begin{proof}
 For any $v\in \mV$, we denote by $\overline{v}=\sum_{n=1}^N [\phi_n,v]\chi_n$, the GP MAP  solution after observing the values $[\phi_n,v]$ for $1\leq n \leq N$. By definition, $\overline{v}\in \mV$ and
  $[\bphi,v-\overline{v}]=0$. Thus, $v-\overline{v}$ is orthogonal to the linear span of $\chi_n$ in the $\|\cdot\|$ norm. It follows that
 $\|v-u^\dagger\|^2=\|v-\overline{v}\|^2+\|\overline{v}-u^\dagger\|^2$. Furthermore writing $\bz=[\bphi,v]$, we have
\begin{equation*}
 \|\overline{v}-u^\dagger\|^2=\|\overline{v}\|^2+\|u^\dagger\|^2-2\bz^T\Theta^{-1}\bz^\dagger =
 \|\overline{v}\|^2-\|u^\dagger\|^2- 2 (\bz-\bz^\dagger)^T\Theta^{-1}\bz^\dagger \leq \|\overline{v}\|^2-\|u^\dagger\|^2,
\end{equation*}
 where we used the relation $\|u^\dagger\|^2 = (\bz^\dagger)^T\Theta^{-1}\bz^\dagger$ and the assumed condition \eqref{eqconditionF}.
 Thus, we deduce that $\|v-u^\dagger\|^2\leq \|v-\overline{v}\|^2 + \|\overline{v}\|^2-\|u^\dagger\|^2 = \|v\|^2 - \|u^\dagger\|^2$ according to the orthogonality. This leads to the second point of the proposition, and the first point follows directly.

 The proof of the third point is similar to that of \cite[Thm.~18.2]{owhadi2019operator}. Here we shall prove that for the given minimax loss, the optimal value is 1 and the minimizer is $\sum_{n=1}^N z_n\chi_n$ for any $\bz$ satisfying $F(\bz)=\by$. To achieve this, first observe that the loss is always less or equal to 1 because for $v\in \mV$, $\|v-\overline{v}\|^2/ \|v\|^2\leq 1$ and $\overline{v} \in \mV$. Thus, the optimal value is not larger than 1.

 Moreover, let $v^\perp$ be a nonzero element in $\mU$ such that $[\bphi, v^\perp]=0$. Consider $v_c = c v^\perp+\sum_{n=1}^N z_n\chi_n $ for any $\bz$ satisfying $F(\bz)=\by$.
 Taking $u = \sum_{n=1}^N z_n\chi_n$, we have that
\[ \frac{\|v_c-u\|^2}{\|v_c\|^2}=\frac{c^2\|v^\perp\|^2}{c^2\|v^\perp\|^2+\|u\|^2}\, ; \]
 this term converges towards $1$ as $c \rightarrow \infty$. Thus, we get that the optimal value is exactly 1 and $u = \sum_{n=1}^N z_n\chi_n$ are minimizers.
\end{proof}

Note that the above proposition implies the uniqueness of the minimizer $u^\dagger$ but
does not imply the uniqueness of the solution $u^\star$. It is also interesting to observe that
condition~\eqref{eqconditionF} only involves the nonlinear map $F$ and not the differential operators defining the PDE.

  \subsection{Error Estimates}\label{sec:error-estimates}
  Here we present rigorous error estimates  in the PDE setting of Subsection~\ref{sec:PDEs} by bounding
  the error between the minimizer $u^\dagger$ of \eqref{PDE-optimization-form-initial} and the
  true solution $u^\star$. The main idea behind our error bounds is to connect the conditional covariance of
  the underlying GP to the pointwise error of  the conditional MAP estimator.

 In the setting of Section \ref{sec:rep-theorems-and-constraints},
 consider  $\ker(\bphi):= \{ u \in \mU  | [\bphi,u] = 0\}$. Then the conditional covariance operator
 $\C_\bphi$ of Proposition \ref{prop:conditional-GP-with-linear-measurements}
 coincides with  the short of the
    operator $\C$ to $\ker(\bphi)$ \cite{anderson1975shorted, owhadi2015conditioning}: this is a symmetric and positive
    operator from $\mU^\ast$ to $\mU$ identified via
    \begin{equation}\label{eqkphi}
      [ \varphi , \C_\bphi \varphi ]  =
      \inf_{\overline{\varphi} \in \text{span }  \bphi } \langle \varphi - \overline{\varphi}, \varphi - \overline{\varphi} \rangle_{\ast} =
      \inf_{\overline{\varphi} \in \text{span } \bphi } [ \varphi - \overline{\varphi}, \C (\varphi - \overline{\varphi})], \qquad \forall \varphi \in \mU^\ast,
    \end{equation}
    where we used the shorthand notation $\text{span } \bphi = \text{span} \{ \phi_1, \dots, \phi_N\}$. This identification holds true even when the $\phi_n$ are not linearly independent.
 With the notion of the shorted operator at hand
we can now present an abstract error bound on the action of dual elements on
 $u^\dagger$ and $u^\star$.

\begin{proposition}\label{thmerrorest}
Let $\mH$ be a Banach space endowed with a quadratic norm such that  the embedding $\mU\subset \mH$ is compact and such that $\mH^\ast \subset \mU^\ast$ where $\mH^\ast$ is defined using the same duality pairing as for $\mU^\ast$.
  Let $u^\dagger=\sum_{n=1}^N z_n^\dagger \chi_n$ be a minimizer of \eqref{PDE-optimization-form-initial} approximating
  the unique solution $u^\star$ of \eqref{prototypical-PDE}.
  Fix $\varphi \in \mH^\ast$ and  define  $\bphi$  as in \eqref{phi-PDE-def}. Then it holds true that
  \begin{equation}\label{eqweak}
    \big| [\varphi, u^\dagger] - [\varphi, u^\star] \big| \le \sigma(\varphi) \| u^\star\|
    + \epsilon \|\varphi\|_{\mH^\ast}  ,
  \end{equation}
  where  $\sigma^2(\varphi) := \K(\varphi,\varphi) -  \K(\varphi,\bphi)\K(\bphi,\bphi)^{-1}\K(\bphi,\varphi)$ and
  \begin{equation*}
    \epsilon^2:= \left\{
    \begin{aligned}
       & \max_{\bz \in \R^N}  && ( \bz - \bz^\dagger )^T A (\bz - \bz^\dagger) \\
       & \st  && F(\bz)  = \by, \\
       & && \bz^T \Theta^{-1} \bz \le \| u^\star\|^2.
    \end{aligned}
    \right.
  \end{equation*}
  where $A$ is the $N\times N$ positive symmetric matrix with entries   $A_{i,j}=\langle \chi_i,\chi_j\rangle_\mH$. Furthermore, it also holds true that
  \begin{equation}\label{eqstrong}     \|u^\dagger-u^\star\|\leq \inf_{\phi \in  \mathrm{span}\bphi}\|u^\star-\C \phi\|+\epsilon\,.
  \end{equation}
\end{proposition}

\begin{proof}
  Define $\bv := [ \bphi, u^\star ]$ and     $ \overline{u} := \bv^T \Theta^{-1} \mK \bphi \in \mU$.
  By  triangle inequality we have
  \begin{equation*}
    \big| [ \vphi, u^\star] - [\vphi, u^\dagger] \big| \le
    \big| [ \vphi, u^\star] - [\vphi, \overline{u}]  \big| + \big| [\vphi, \overline{u}] - [\vphi, u^\dagger] \big|.
  \end{equation*}
  Since $[\bphi, u^\star]  = [\bphi, \overline{u}]$ then by \cite[Thm.~5.1]{Owhadi:2014},
  \begin{equation*}
  \big| [ \vphi, u^\star] - [\vphi, \overline{u}] \big|=
  \inf_{\overline{\varphi}\in \text{span}\bphi}
  \big|[\vphi-\overline{\varphi},u^\star-\overline{u}]\big|
  \leq \inf_{\overline{\varphi}\in \text{span}\bphi} \|\vphi-\overline{\varphi}\|_* \|u^\star-\overline{u}\|.
  \end{equation*}
  Then  \eqref{eqkphi} implies that $\inf_{\overline{\varphi}\in \text{span}\bphi} \|\vphi-\overline{\varphi}\|_*=\sigma(\varphi)$.
  Since $\overline{u}$ minimizes $\|u\|$ subject to $[\bphi,u]=[\bphi,u^\star]$, we have $\|u^\star-\overline{u}\|^2=
  \|u^\star\|^2-\|\overline{u}\|^2 $, which leads to
    $| [ \vphi, u^\star] - [\vphi, \overline{u}] | \le \sigma(\vphi) \| u^\star \|.$
  Furthermore,
  $| [\vphi, \overline{u}] - [\vphi, u^\dagger] |\leq
  \|\vphi\|_{\mH^\ast}
  \|\overline{u}-u^\dagger\|_\mH$ and observe that
  $\|\overline{u}-u^\dagger\|_\mH^2=(\bz-\bz^\dagger)A (\bz-\bz^\dagger) $.
  Since
  $\| \overline{u} \| \le \|u^\star\|$ and $\| u^\dagger \| \le \| u\|^\star$, and $
  F([ \bphi, \overline{u}]) =  \by$ we deduce that $\| \overline{u} - u^\dagger \|_{\mH}$ is bounded by the supremum of
  $\| v' - u^\dagger \|_{\mH}$ over all $v' \in \mU$ that belong to the ball of radius $\| u^\star \| $ and satisfy
  the PDE constraints at the collocation points. Inequality \eqref{eqstrong} follows from $\|u^\dagger-u^\star\|\leq  \|u^\star-\overline{u}\|+\epsilon$ and the the fact that $\overline{u}$ is the projection of $u^\dagger$ onto $\text{span}\: \bphi$ \cite[Thm.~12.3]{owhadi2019operator}.
\end{proof}

In the case where $\mH \subset C^{t}(\Omega) \cap C^{t'}(\overline{\Omega})$
we can take $\varphi$ to be of the form $\updelta_x \circ L$ where $L$ is a differential operator of order at most $t$ for $x\in \Omega$, and $t'$ for $x\in \partial \Omega$. Then Proposition~\ref{thmerrorest} allows us to
bound the poitwise error of $u^\dagger$ and its derivatives.
As in Theorem \ref{thm:convergence}, the compact embedding of  $\mU$ into $\mH$ guarantees the convergence of $\sigma(\varphi)$ and $\epsilon$ towards zero as the collocation points become dense in $\Omega$ and $\partial \Omega$.
Note that, by \eqref{eqstrong}, the strong approximation error between $u^\dagger$ and $u^\star$ is controlled by the sum between $\epsilon$ and the best approximation error in $\text{span} \: \C\bphi$. Controlling this best approximation error is typically done in the dual space in the linear setting \cite[Sec.~14.8]{owhadi2019operator}. In the nonlinear setting, it requires  controlling the regularity of the solution (analyzing the underlying PDE). Moreover if the fill distance of the collocation points is $h$, and if the norm $\|\cdot\|$ satisfies the embedding inequality $\|\cdot\|\leq C \|\cdot\|_{H^s(\Omega)}$, then \cite[Lem.~14.39]{owhadi2019operator} implies
$\|u^\star-\overline{u}\|\leq C h^s \| u^\star\|_{H^{2s}(\Omega)}$.

\subsection{Learning the Kernel}\label{sec:KernelFlow}
An important factor in the accuracy and performance of our method is the choice of the kernel $\cc$ (and in turn, the operator
$\C$).  The importance of the kernel choice is readily visible in Theorem~\ref{thm:convergence} through the
assumption that $u^\star$, the solution of the PDE, should belong to the  $\mU$ the RKHS corresponding to $\cc$.
As discussed in \cite{Owhadi:2014}, even for linear PDEs, if the underlying kernel is not adapted to the solution space of the PDE, then the resulting method can perform arbitrarily badly \cite{babuvska2000can}.
This point stresses the importance of learning (or possibly programming \cite{owhadi2019kernelkmd, owhadi2020ideas}) an informed kernel beyond enforcing the PDE at collocation points.
While the problem of selecting the kernel is well-understood  if  the solution is smooth, or  the PDE is linear
with possibly rough coefficients \cite{Owhadi:2014}, it can be challenging when the underlying PDE is nonlinear, and the solution is irregular.
Since our  approach employs an explicit kernel it allows us to learn the kernel directly
using kernel (hyperparameter) learning techniques such as maximum likelihood estimation (MLE) \cite{chen2020consistency}, MAP estimation  \cite{owhadi2020ideas}, or cross-validation (CV)
\cite{friedman2001elements, owhadi2019kernel, chen2020consistency}. Adaptation of these techniques for solution of nonlinear PDEs is an
attractive direction for future research.

\subsection{Non-Gaussian priors}\label{subsecnongaussian}
Although we have employed a Gaussian prior in the design of our nonlinear solver and in our Bayesian IPs, our methodology can naturally be generalized to  priors that can be obtained as nonlinear transformations of Gaussian measures (e.g., instead of replacing the solution $u^\star$ of a nonlinear PDE by a GP $\xi$, one could replace it by $G(\xi)$ where $G$ is a nonlinear function with the appropriate degree of regularity so that the pointwise derivatives of $G(\xi)$ remain well-defined). This approach can potentially lead to efficient methods for computing MAP estimators of Bayesian IPs with certain non-Gaussian priors,
a contemporary problem in the uncertainty quantification and IP literatures.

\section*{Acknowledgments}
The authors gratefully acknowledge support by  the Air Force Office of Scientific Research under MURI award number FA9550-20-1-0358 (Machine Learning and Physics-Based Modeling and Simulation).
HO also acknowledges support
by  the Air Force Office of Scientific Research under award number FA9550-18-1-0271 (Games for Computation and Learning). YC is also grateful for support from the Caltech Kortchak Scholar Program.

\bibliographystyle{siamplain}
\bibliography{GPfornonlinearPDEs}

\appendix

\section{Diagonal Regularization of Kernel Matrices} \label{app:reg-theta}
\as{In the context of GP regression it is common to regularize
the kernel matrices by addition of a small diagonal term;
in that context, doing so has the interpretation of assuming
small Gaussian noise on top of the observations. This diagonal
regularization is sometimes referred to as a ``nugget''.
Here we discuss a related approach to regularizing kernel matrices ($\Theta$ and $\tTheta$)
by perturbing them slightly; for brevity we use the terminology
``nugget'' throughout.} \bh{In Appendix~\ref{sec:nuggets} we present an adaptive
approach to constructing a family of nugget terms that is tailored to our
kernel matrices. Appendices~\ref{app:Nonlinear-Elliptic-Nugget} through \ref{app:Eikonal-Nugget}
gather the detailed choices of nugget terms for the experiments in
Subsections~\ref{sec:elliptic-PDE} through \ref{sec:eikonal-PDE}.
Appendix~\ref{app:Darcy-Nugget} contains the same details for the
experiments in Subsection~\ref{sec:Darcy-flow}.
}
\subsection{An Adaptive Nugget Term}\label{sec:nuggets}
\bh{One of the main computational issues in implementing our methodology is
that the kernel  matrix $\Theta$ is ill-conditioned.} As a consequence,  we need to regularize this matrix to improve the stability of these algorithms.
This regularization may introduce an accuracy floor, so it is important to
choose the regularization term so that it has a small effect on accuracy -- there is thus an accuracy-stability tradeoff.
A traditional strategy for achieving this goal is to add a nugget term $\eta I$ to $\Theta$, where $\eta > 0$ is a small number, and $I$ is the identity matrix. By choosing a suitable $\eta$, the condition number of $\Theta + \eta I$ can be improved
significantly. However, there is an additional level of difficulty in our method since  the matrix $\Theta$ contains multiple blocks whose spectral properties can differ
\as{by orders of magnitude}, since they can involve different orders of derivatives of the kernel function $\cc$. This observation implies that adding $\eta I$, which is uniform across all blocks, may be suboptimal in terms of the accuracy-stability tradeoff.

\bh{
In what follows we adopt the same notation as Subsection~\ref{sec:constructing-theta}.
To address the ill-conditioning of $\Theta$,} we consider adding an adaptive block diagonal
nugget term. More precisely, without loss of generality we can assume $\Theta^{(1,1)}$ corresponds to the pointwise measurements, i.e., $L_1^{\bx} = \updelta_{\bx}$, then, we construct a block diagonal matrix
\begin{equation*}
  R =
  \begin{bmatrix}
    I &  &   &  \\
     & \frac{\tr(\Theta^{(2,2)})}{\tr(\Theta^{(1,1)})}I &  &  \\
     &  & \ddots & \\
     &  &  & \frac{\tr(\Theta^{(Q,Q)})}{\tr(\Theta^{(1,1)})}I
  \end{bmatrix}\, ,
\end{equation*}
where we reweight the identity matrix in each diagonal block by a factor of the trace ratio between $\Theta^{(q,q)}$ and $\Theta^{(1,1)}$. With this matrix, the adaptive nugget term is defined as $\eta R$ with a global nugget parameter $\eta > 0$.
We find that once the parameter $\eta$ is chosen suitably, then our  \bh{Gauss--Newton}  algorithm converges quickly and in a
stable manner. During computations, we can compute the  Cholesky factors of $\Theta + \eta R$ offline and use
back-substitution to invert them in each iteration of \bh{Gauss--Newton.}

\begin{runningexample}{NE}
For the numerical experiments in Subsection~\ref{sec:intro:implementation} pertaining to the
 nonlinear elliptic PDE \eqref{elliptic-proto-PDE}, we observed that  $\Theta$ has
 two distinct diagonal blocks, i.e., one block corresponding to the pointwise evaluation functions
 and with entries $K(\bx_m, \bx_i)$ and another block corresponding to pointwise evaluations of
 the Laplacian operator and with entries $\Delta^\bx \Delta^{\bx'} K( \bx, \bx') |_{(\bx_m, \bx_i)}$.
With $M=1024$ collocation points, the trace ratio between these blocks  was of order $4000$. Thus, the difference between $I$ and $R$ is significant. Our experiments also showed that if we only add $\eta I$ to regularize the matrix, \as{then
 choosing $\eta$  as large as  $O(10^{-4})$ was} needed to get meaningful results, while using the nugget term $\eta R$, we could choose $\eta = 10^{-9}$
 which leads to significantly improved results. We further explore these details below
 and in particular in Table~\ref{table: Comparison between traditional nugget terms and our adaptive nugget terms}.
\end{runningexample}

\subsection{Choice of Nugget Terms for the Nonlinear Elliptic PDE}
\label{app:Nonlinear-Elliptic-Nugget}
\bh{Below we discuss the choice of the nugget term in the numerical experiments of Subsection~\ref{sec:elliptic-PDE}. The results in Figure~\ref{fig:nonlinear elliptic, demon}
and Table~\ref{Table:Comparison between the elimination and relaxation approaches}
were obtained by employing the adaptive nugget term of Appendix~\ref{sec:nuggets} with global parameters $\eta = 10^{-13}$
and $\eta=10^{-12}$ respectively.

We  also compared our adaptive nugget term to more
standard choices, i.e.,  nugget
terms of the form $\eta I$ against our adaptive nugget term $\eta R$ with the nonlinearity $\tau(u) = u^3$}.
Cholesky factorization was applied to the regularized matrix and the subsequent \bh{Gauss--Newton}  iterations were employed to obtain the solutions. The $L^2$ and $L^{\infty}$ errors of the converged solutions are shown in Table \ref{table: Comparison between traditional nugget terms and our adaptive nugget terms}. The results were averaged over 10 instances of a random sampling of the collocation points.
Here, ``nan'' means the algorithm was unstable and diverged.
\as{To obtain these results we terminated all Gauss-Newton iterations after 5 steps.
Due to randomness in the initial guess, and in examples where random collocation
points were used due to this too,} \bh{we observed
that some random trials did not converge in 5 steps. This variation also explains the non-monotonic behavior of the error in Table~\ref{table: Comparison between traditional nugget terms and our adaptive nugget terms} as $\eta$ decreases. These effects were more
profound  for the standard nugget term.
Besides these small variations our results clearly
demonstrate the superior accuracy and stability that is provided by our adaptive nugget
 term versus the standard nugget choice.}

\begin{table}[ht]
\centering
\yc{
\begin{tabular}{lllllll}
\hline
 $\eta$ & $ 10^{-1}$ & $10^{-2}$ & $ 10^{-3}$ & $ 10^{-4}$ & $ 10^{-5}$ & $ 10^{-6}$ \\ \hline
$\Theta + \eta I$: $L^2$ error      & 7.77e-02       & 4.46e-02 & 2.65e-02             & 1.56e-02             & 1.32e-02           & 1.46e-02                  \\
$\Theta + \eta I$: $L^\infty$ error & 6.43-01    & 3.13e-01             & 1.99e-01              & 1.47e-01              & 1.33e-01              & 1.43e-01
\\
$\Theta + \eta R$: $L^2$ error      & 8.49e-02      & 9.29e-03    & 9.10e-03     & 3.34e-03   & 1.01e-03  & 3.36e-04      \\
$\Theta + \eta R$: $L^\infty$ error & 4.02e-01        & 7.86e-02     & 5.58e-02       & 2.21e-02      & 7.17e-03    & 3.87e-03      \\ \hline
 $\eta$ & $ 10^{-7}$ & $10^{-8}$ & $ 10^{-9}$ & $ 10^{-10}$ & $ 10^{-11}$ & $ 10^{-12}$ \\ \hline
$\Theta + \eta I$: $L^2$ error      & 1.37e-02  & 8.623-03      & 1.01e-02           & 1.92e-02              & nan              & nan                 \\
$\Theta + \eta I$: $L^\infty$ error & 1.81e-01    & 8.28e-02              & 1.07e-01             & 3.05e-01              & nan              & nan
\\
$\Theta + \eta R$: $L^2$ error      & 1.55e-04       & 7.05e-05     & 4.56e-05    & 6.30e-06   & 1.73e-06  & 8.31e-07     \\
$\Theta + \eta R$: $L^\infty$ error & 2.41e-03        & 1.07e-03     & 7.66e-04       & 8.92e-05      & 2.62e-05   & 1.17e-05      \\ \hline
\end{tabular}}
\label{table: Comparison between traditional nugget terms and our adaptive nugget terms}
\caption{Comparison of solution errors between standard nugget terms and our adaptive nugget terms for the nonlinear elliptic PDE \eqref{elliptic-proto-PDE}. Collocation points are uniformly sampled with $M=1024$ and $M_{\Omega}=900$ with a Gaussian kernel with
lengthscale parameter $\sigma=0.2$. Results are averaged over 10 realizations of the
random collocation points. \yc{The maximum Gauss-Newton iteration was  5.}}
\end{table}

\subsection{Choice of Nugget Terms for Burger's Equation}
\label{app:Burgers-Nugget}
\bh{
For the numerical experiments in Subsection~\ref{sec:Burgers-Revisited}
we primarily used our adaptive nugget term
 as outlined in Appendix~\ref{sec:nuggets}. For the results in Figure~\ref{fig: Burgers demonstration}
we used a global nugget parameter \yc{$\eta = 10^{-10}$}.  For the convergence
analysis in Table~\ref{table: Burgers convergence} we varied $\eta$ for different number
of collocation points to achieve better performance. More precisely, for $M \le 1200$ we
used a larger nugget $\eta = 10^{-5}$ and for $M \ge 2400$ we used
$\eta = 10^{-10}$. Choosing a smaller nugget for small values of $M$
would still yield equally accurate results but required more iterations of the
Gauss-Newton algorithm.}

\subsection{Choice of Nugget Terms for the Eikonal Equation}
\label{app:Eikonal-Nugget}
\bh{Our numerical experiments in Subsection~\ref{sec:eikonal-PDE}
also used the adaptive nugget of Appendix~\ref{sec:nuggets}. Indeed, we followed
a similar approach to choosing the global parameter $\eta$ as in the case of Burger's equation
outlined in Appendix~\ref{app:Burgers-Nugget} above.

For the results in Figure~\ref{fig:Eikonal} we used $\eta = 10^{-10}$. For the
convergence analysis in Table~\ref{table: Eikonal equation, error} we varied $\eta$ for
different values of $M$, i.e., we chose $\eta = 10^{-5}$ for $M \le 1200$ and $\eta = 10^{-10}$
for $M \ge 2400$. Analogously to the Burger's experiment we observed that smaller values
of $\eta$ for smaller choices of $M$ cause the Gauss-Newton iterations to converge more slowly.
Thus varying $\eta$ with $M$ improved the efficiency of our framework.
}

\subsection{Choice of Nugget Terms for Darcy Flow}
\label{app:Darcy-Nugget}
\bh{
    Both of the matrices $\Theta, \tTheta$
    outlined in Subsection~\ref{sec:IP-setup}
    are dense and ill-conditioned in the IP setting and so an appropriate
    nugget term should be chosen to regularize them. In the IP setting we propose to add adaptive nuggets for
    both $\Theta, \tTheta$ using the same strategy as in Appendix~\ref{sec:nuggets}, except that
    the nuggets are constructed independently for each matrix. To this end we set $\Theta \leftarrow \Theta + \eta R$
    and $\tTheta \leftarrow \tTheta + \tilde{\eta} \tilde{R}$, where the $\tilde{\eta}, \tilde{R}$ denote the
    global nugget parameter and the re-weighted identity matrix corresponding to $\tTheta$.
   For the numerical experiments in Figure~\ref{fig:Darcy flow} we used $\eta = \tilde \eta = 10^{-5}.$
}

\end{document}